\documentclass[a4paper,reqno]{amsart}

\usepackage[
left=2.8cm,
right=2.8cm,
top=3cm,
bottom=3cm
]{geometry}

\usepackage{amsfonts}
\usepackage{amssymb}
\usepackage{amsmath}
\usepackage{mathrsfs}
\usepackage{amsthm}
\usepackage[colorlinks,linkcolor=blue,citecolor=blue]{hyperref}
\usepackage{empheq}
\usepackage{microtype}
\usepackage[normalem]{ulem}
\usepackage{marginnote}

\def\red{\color{red}}

\newtheorem{theorem}{Theorem}[section]
\newtheorem{corollary}[theorem]{Corollary}
\newtheorem{lemma}[theorem]{Lemma}
\newtheorem{proposition}[theorem]{Proposition}

\theoremstyle{remark}
\newtheorem{remark}[theorem]{Remark}

\theoremstyle{definition}

\numberwithin{equation}{section}

		\title[Classification of solutions to a class of $N$-Liouville equations]{On the classification of solutions to a class of $N$-Liouville equations in $\mathbb{R}^N$}

		\author[G. Ciraolo]{Giulio Ciraolo}
		\address[G. Ciraolo]{Dipartimento di Matematica ``Federigo Enriques"\\
			Universit\`a degli Studi di Milano\\ Via Cesare Saldini 50, 20133 Milan\\
			Italy}
		\email{giulio.ciraolo@unimi.it}

		\author[P. Esposito]{Pierpaolo Esposito}
		\address[P. Esposito]{Dipartimento di Matematica e Fisica\\
			Universit\`a degli Studi Roma Tre\\ Largo S. Leonardo Murialdo 1,
			00146 Roma\\ Italy}
		\email{pierpaolo.esposito@uniroma3.it}

		\author[X. Li]{Xiaoliang Li}
		\address[X. Li]{Dipartimento di Matematica ``Federigo Enriques"\\
			Universit\`a degli Studi di Milano\\ Via Cesare Saldini 50, 20133 Milan\\
			Italy}
		\email{xiaoliang.li@unimi.it}
		
		\thanks{}
		
		
\begin{document}
		
			\begin{abstract}
			Given $N\geq 2$ and $\alpha>-1$, we consider the following  weighted Liouville-type equation involving the $N$-Laplacian:
			\begin{equation*}
				\begin{cases} 
					- \Delta_N u = |x|^{N\alpha} e^u & \text{ in } \mathbb{R}^N, \\
					\int_{\mathbb{R}^N} |x|^{N\alpha} e^u \, dx < + \infty\,.
				\end{cases}
			\end{equation*}
		Solutions are completely classified by Prajapat and Tarantello \cite{PT01} when $N=2$ via complex analysis (see also \cite{CW1994}), and by the second author when $\alpha=0$ using Pohozaev identities and an isoperimetric argument \cite{Esposito18}. 
		
		In this paper, we first devise a $P$-function approach to the classification result  \cite{PT01} for all $\alpha>-1$ when $N=2$.  Since it is not based on complex analysis,  this alternative and more PDE-oriented approach naturally extends to $N\geq 3$ by providing the classification for any $-1<\alpha\leq 0$.  In particular, the explicit radial solutions \eqref{eq:thm-u-HD} are the unique ones for $-1<\alpha\leq0$ but become degenerate for special values $\alpha_k>0$, a hint that non-radial solutions might arise for $\alpha>0$ as it happens when $N=2$.
		\end{abstract}
		
\maketitle

\section{Introduction}
In this paper, we are concerned with the problem
\begin{equation} \label{eq_general}
	\begin{cases} 
		- \Delta_N u = |x|^{N\alpha} e^u & \text{ in } \mathbb{R}^N,\\
		\int_{\mathbb{R}^N} |x|^{N\alpha} e^u \, dx < + \infty
	\end{cases}
\end{equation}
where $\alpha>-1$ is a real number, $N\geq 2$ is an integer and $\Delta_Nu=\operatorname{div}(|\nabla u|^{N-2}\nabla u)$ stands for the $N$-Laplacian operator.

We aim to address classification issues for \eqref{eq_general}. Here, solutions belong to $W_{\mathrm{loc}}^{1,N}(\mathbb{R}^N)$ and satisfy \eqref{eq_general} in a weak sense. It can be seen that \eqref{eq_general} has an explicit radial solution
\begin{equation}\label{eq:U}
	U_{\alpha}(x) :=\log\frac{c_N(\alpha+1)^N}{\left(1+|x|^{\frac{N(\alpha+1)}{N-1}}\right)^N}, \quad c_N=N\left(\frac{N^2}{N-1}\right)^{N-1}.
\end{equation}
By the scaling invariance, a family of radial solutions to \eqref{eq_general} is consequently given by
\begin{equation}\label{eq:thm-u-HD}
	U_{\alpha,\lambda}(x) := U_{\alpha}(\lambda^{\frac{N-1}{N(\alpha+1)}} x)+N (\alpha+1)\log \lambda^{\frac{N-1}{N(\alpha+1)}}=\log\frac{c_N(\alpha+1)^N \lambda^{N-1}}{\left(1+\lambda  |x|^{\frac{N(\alpha+1)}{N-1}}\right)^N}
\end{equation}
for $\lambda>0$. We are interested in whether the family $U_{\alpha,\lambda}$ exhausts all the solutions to \eqref{eq_general}. 


In the semilinear case $N=2$, note that \eqref{eq_general} is equivalent to the Liouville equation
\begin{equation}\label{eq:SLio2}
	\begin{cases} 
	 -\Delta v = e^v - 4\pi\alpha\,\delta_0 & \text{ in } \mathbb{R}^2, \\
		\int_{\mathbb{R}^2} e^v \, dx < + \infty
	\end{cases}
\end{equation}
in terms of $v=u+2\alpha\log |x|$, where $\delta_0$ denotes the Dirac measure at the origin. Problems of type \eqref{eq:SLio2} arise in conformal geometry and various branches
of physics and have been the object of many studies in recent years. In particular, analytical aspects related to \eqref{eq:SLio2} have been discussed in the context of blowup analysis and variational/topological properties; see for instance \cite{BCLT04,BT17,BM91,CGS12,CL03,CC15,Paolo25,DEM12,Li99,LS94,PT01,Tar04,Tar005,Tar05,WZ21,WZ26,Zhang09} and the references therein. The cases $\alpha=0$ and $\alpha \not=0$ are referred to as the regular and singular Liouville equation, respectively.

Here, in the regular case the first classification result for \eqref{eq:SLio2} is a direct consequence of the formula established by Liouville in \cite{Liou} on simply-connected domains. The first proof based on PDEs techniques is due to Chen and Li \cite{CL1991},  which establishes the radial symmetry of every solutions to \eqref{eq:SLio2}$_{\alpha=0}$ via the method of moving planes.  All solutions are then given by \eqref{eq:thm-u-HD} with $N=2$, up to translations (recall the translation invariance of \eqref{eq:SLio2}$_{\alpha=0}$ in this case).

Afterwards, in the context of the singular Liouvile equation the above result was extended to the case $-1<\alpha\leq0$ in \cite{CL95} using the same method, and was also revisited in \cite{BLD04,CK1995,HW06} by different techniques.  However, the radial symmetry is \emph{broken}  when $\alpha\in\mathbb{N}\setminus\{0\}$ since \eqref{eq:SLio2} also admits \emph{nonradial} solutions, as explicitly exhibited by Chanillo--Kiessling \cite{CK1994}. Nevertheless,  a complete classification of solutions to \eqref{eq:SLio2} for arbitrary $\alpha>-1$ was later obtained by Prajapat and Tarantello \cite{PT01}. To be precise, in what follows we use the complex notation $z=x+iy\in\mathbb{C}$ for $(x,y)\in\mathbb{R}^2$. Thanks to the equivalence between \eqref{eq_general}$_{N=2}$ and \eqref{eq:SLio2}, the following result was derived in \cite{PT01}.
\begin{theorem}\label{thm:2D}
	Let $N=2$ and $\alpha>-1$. Any solution $u$ to \eqref{eq_general} takes the form
	\begin{equation}\label{eq:thm-u}
		u(z)=\log \frac{8(\alpha+1)^2 \lambda^2}{(1+\lambda^2|z^{\alpha+1}+c|^2)^2}
	\end{equation}
	for some $\lambda>0$ and $c\in\mathbb{C}$, with $c=0$ if $\alpha\notin\mathbb{N}$.
\end{theorem}

When $\alpha\in\mathbb{N}\setminus\{0\}$,  note that the function \eqref{eq:thm-u} is not radially symmetric about any point when $c\neq 0$,  as observed first in \cite{CK1994}.  The proof of Theorem \ref{thm:2D} in \cite{PT01} crucially exploits the extension of the Liouville formula \cite{Liou} to a punctured disc obtained by Chou and Wan \cite{CW1994}.

However, to the best of our knowledge, a PDE proof of Theorem \ref{thm:2D} is still unavailable in literature, as well as a classification result for problem \eqref{eq_general} in the quasilinear case $N\geq 3$ remains completely out of reach, except in the special case $\alpha=0$. Indeed, the second author showed in \cite{Esposito18} that all solutions to \eqref{eq_general}$_{\alpha=0}$ are given by \eqref{eq:thm-u-HD} up to translations. The result in \cite{Esposito18} was achieved by performing first a fine asymptotic analysis at infinity for every solution $u$ and then exploiting a combination of Pohozaev identities with an isoperimetric argument, and was subsequently extended to an anisotropic setting in \cite{CL24}. However, such argument fails for classification issues on \eqref{eq_general} when $\alpha\neq0$ but still allows to determine in \cite{Esposito21} the total mass $\int_{\mathbb{R}^N} |x|^{N\alpha} e^u \, dx$ of any solution $u$ of \eqref{eq_general}, as we will discuss more later.

On the other hand, the Liouville equation
\begin{equation}\label{eq:Lio}
	-\Delta u=e^u \quad\text{ in } M
\end{equation} 
has recently attracted attention in \cite{CaiLai24,CM24,CFP24} on \emph{complete} Riemmanian surfaces $M$ with \emph{nonnegative} Ricci curvature. In particular, the first author, Farina and Polvara \cite{CFP24} established a classification result for \eqref{eq:Lio} under a logarithmic lower bound assumption on the solutions (\emph{instead of} the finite-mass condition) and also derived the resulting rigidity of $M$.  Simultaneously, they obtained parallel results for the critical equation on Riemmanian manifolds of dimension $N\geq 3$. These results were obtained in \cite{CFP24} through a unified approach based on the introduction of a suitable $P$-function. Very recently, this approach has also been used in \cite{CP25} to study the classification of extremals for the Caffarelli--Kohn--Nirenberg inequalities, and has been exploited \emph{quantitatively} in \cite{CG25} to address the stability of classification results for solutions to the critical $p$-Laplace equation. We also refer to \cite{CFP25} for a recent extension of this approach to some more general settings.

Inspired by \cite{CFP24,CG25,CP25}, in the present paper we investigate problem \eqref{eq_general} by exploiting a $P$-function approach, which enables us to completely classify its solutions for any $-1<\alpha\leq 0$ when $N\geq 3$ as stated in the following result.
\begin{theorem} \label{thm:HD}
	Let $N\geq 3$ and $-1<\alpha\leq0$. Let $u\in W_{\mathrm{loc}}^{1,N}(\mathbb{R}^N)$ be a solution to \eqref{eq_general}. Then $u$ is of the form $U_{\alpha,\lambda}$ defined in \eqref{eq:thm-u-HD} for some $\lambda>0$ (up to a translation when $\alpha=0$).
\end{theorem}
In particular, for $\alpha=0$ we provide an alternative proof of the above-mentioned result \cite{Esposito18}.  Moreover,  the same approach also works in the case $N=2$ and allows us to obtain a complete classification of \eqref{eq_general}$_{N=2}$ in the full range $\alpha>-1$, thereby yielding an alternative proof of Theorem \ref{thm:2D} from a more PDE-oriented viewpoint.

The proofs rely on introducing a suitable $P$-function associated with the solution $u$ and demonstrating its constancy, which yields the desired classification of $u$. To this aim, we start with a conformal reformulation of problem \eqref{eq_general}, which removes the presence of the weighted term $|x|^{N\alpha}$ and turns \eqref{eq_general} into the prototypical form as represented by \eqref{eq:Lio} when $N=2$. More precisely, let us define the conformal metric
 \begin{equation}\label{eq:def-g}
 	g=|x|^{2\alpha}\delta_e \,,
 \end{equation}
with $\delta_e$ the standard Euclidean metric on $\mathbb{R}^N$. It is derived from \eqref{eq_general} that $$-\Delta_N^g u =e^u \quad\text{ in }\mathbb{R}^N\setminus\{0\},$$
where $\Delta_N^g$ stands for the $N$-Laplace operator with respect to the metric $g$ (see \eqref{eq:N-Lap-g-u}). Then by letting $v=e^{-\frac{u}{N}}$, we find that the desired $P$-function can be defined as
\begin{equation}\label{eq:def-P}
	P:=\Delta_N^g v \,.
\end{equation}
Indeed, we deduce that $P$ given by \eqref{eq:def-P} satisfies a differential identity involving the Ricci tensor in the (\emph{noncomplete}) Riemannian manifold $(\mathbb{R}^N\setminus\{0\}, g)$; see Lemma \ref{lem:identity-P}. This identity can be further translated into an integral inequality as shown in Lemma \ref{lem:intergal-idenitity-Pt}. Then by taking suitable cut-off functions in this inequality, we conclude in Proposition \ref{pro:P0} that $P$ is constant. Finally, we prove that the constancy of $P$ totally determines the solution $u$ to take the form \eqref{eq:thm-u-HD} (or \eqref{eq:thm-u}) under the assumptions of Theorem \ref{thm:HD} (or Theorem \ref{thm:2D}).

 Notably, it is necessary to assume $\alpha>-1$ in order for problem \eqref{eq_general} to admit a solution, whenever $N\geq 2$. This was proved in \cite[Theorem 1.4]{Esposito21} as a consequence of a quantization property established there, which states that any solution $u$ to \eqref{eq_general} satisfies
\begin{equation}\label{eq:quantization}
	\int_{\mathbb{R}^N} |x|^{N\alpha} e^u\,dx=c_N \omega_N (\alpha+1)^{N-1}
\end{equation}
where $\omega_N$ denotes the volume of the unit ball in $\mathbb{R}^N$. Property \eqref{eq:quantization} imposes a rigid constraint on the solvability of \eqref{eq_general}, suggesting the presence of an underlying structural feature of the solution $u$ whenever $\alpha>-1$. 

However, in the classification arguments described above, the nonnegativity of the Ricci curvature $\mathrm{Ric}_g$ of the metric $g$ plays an essential role. This condition always holds when $N=2$ (in fact, $\mathrm{Ric}_g\equiv 0$ in this case), but it requires $\alpha\leq 0$ whenever $N\geq 3$, as can be seen explicitly from \eqref{eq:Ric}. Accordingly, in Theorem \ref{thm:HD} we restrict to the range $-1<\alpha\leq 0$.

It is natural to ask whether the classification result in Theorem \ref{thm:HD} can be extended to the case $\alpha>0$. Recall that when $N=2$, the structure of solutions to \eqref{eq_general} changes drastically, passing from the necessity of  the radial family (when $\alpha\notin\mathbb{N}\setminus\{0\}$) to the occurrence of a nonradial branch (when $\alpha\in\mathbb{N}\setminus\{0\}$), as shown in Theorem \ref{thm:2D}. To understand whether such a nonradial bifurcation may persist in higher dimensions, we study the linearized problem associated with \eqref{eq_general} (see \eqref{eq:Linearized}) and characterize \emph{all} its solutions; see Theorem \ref{thm:linearized}. This result shows that the radial solution $U_\alpha$ given by \eqref{eq:U} is \emph{degenerate} when $\alpha=\alpha_k>0$ for some $k\in\mathbb{N}\setminus\{0,1\}$, where
\begin{equation}\label{eq:alpha-k}
	\alpha_k=\sqrt\frac{k^2+(N-2)k}{N-1} -1 \,,
\end{equation}
in the sense that the elements in the kernel of the linearized operator are \emph{not only} those arising from the invariance of problem \eqref{eq_general} under scaling. Such degeneracy may indicate the existence of nonradial solutions to \eqref{eq_general} for these special values of $\alpha$ given by \eqref{eq:alpha-k}, as indeed occurs when $N=2$. However, a further investigation of this issue goes beyond the scope of the present paper and will be addressed in our future work.

The remainder of the paper is organized as follows. Section \ref{sec:pre} collects auxiliary results concerning the regularity and asymptotic behavior of solutions to \eqref{eq_general}. In Section \ref{sec:P}, we analyze the properties of the $P$-function introduced in \eqref{eq:def-P}. Then in Section \ref{sec:pf}, we provide the proofs of Theorems \ref{thm:2D} and \ref{thm:HD}. Finally, in Section \ref{sec:linear}, we study the linearized problem associated with \eqref{eq_general}.

\section{Preliminaries} \label{sec:pre}
We present here some results about the regularity and asymptotic behavior at infinity (and at the origin when $N=2$) of solutions to \eqref{eq_general}, which will be useful in the following sections. In particular, we obtain in Proposition \ref{pro:refine-gradient} a fine decay estimate for the gradient of the solutions near the origin when $N=2$ and $\alpha>0$. Throughout the section, we let $u$ be a solution to \eqref{eq_general}.

\begin{lemma}\label{lem:pre}
Let $N\geq2$ and $\alpha>-1$. Then $u\in L^\infty_{\mathrm{loc}}(\mathbb{R}^N)\cap C^{1,\beta}_{\mathrm{loc}}(\mathbb{R}^N\setminus\{0\})$ for all $\beta\in(0,1)$ and $|\nabla u|^{N-2}\nabla u\in W_{\mathrm{loc}}^{1,2}(\mathbb{R}^N\setminus\{0\})$. Also, if $\alpha>-\frac1N$, then $u\in C^{1,\beta}_{\mathrm{loc}}(\mathbb{R}^N)$; if $-1<\alpha\leq-\frac{1}{N}$, then 
\begin{equation}\label{eq:KM-gradient}
	|\nabla u|^{N-1}\in L^{\frac{Np}{N-p}}_{\mathrm{loc}}(\mathbb{R}^N)\quad\text{ for any }1<p<-\frac{1}{\alpha} \,.
\end{equation}
Moreover, there hold
\begin{equation}\label{eq:asym-u}
u+\frac{N^2(\alpha+1)}{N-1}\log |x| \in L^{\infty}(B_1^c)
\end{equation}
and 
\begin{equation}\label{eq:asym-grad}
	\lim _{|x| \rightarrow +\infty}\left(|x| \nabla u(x)+ \frac{N^2(\alpha+1)}{N-1}\frac{x}{|x|}\right)=0\,.
\end{equation}
\end{lemma}

\begin{proof} By \cite[Theorem 1.1]{Esposito21} we know that $u\in L^\infty_{\mathrm{loc}}(\mathbb{R}^N)$, and then $u\in C^{1,\beta}_{\mathrm{loc}}(\mathbb{R}^N\setminus\{0\})$ for all $\beta\in(0,1)$ follows by the classical regularity results in \cite{DiBen1983,Serrin1964,Tolksdorf1984}. The property $|\nabla u|^{N-2}\nabla u\in W_{\mathrm{loc}}^{1,2}(\mathbb{R}^N\setminus\{0\})$ follows from, for instance, \cite[Theorem 1.1]{ACF23}.

If $\alpha>-\frac1N$, then $|x|^{N\alpha}e^u\in L^p_{\mathrm{loc}}(\mathbb{R}^N)$ for some $p>N$, which implies $u\in C^{1,\beta}_{\mathrm{loc}}(\mathbb{R}^N)$ by \cite[Theorem 1.4]{ACF23}.  If $-1<\alpha\leq-\frac{1}{N}$, then $|x|^{N\alpha}e^u\in L^p_{\mathrm{loc}}(\mathbb{R}^N)$ for any $1<p<-\frac{1}{\alpha}\leq N$, so one gets \eqref{eq:KM-gradient} by the regularity result in \cite[Corollary 1--(C7)]{KM14} with both $q$ and $\gamma$ therein taken as $p$ here and by the H\"older's inequality for Lorentz spaces.
	
	Regarding \eqref{eq:asym-u}-\eqref{eq:asym-grad}, they can be found in the proof of \cite[Theorem 1.4]{Esposito21} by taking \eqref{eq:quantization} into account.
\end{proof}

In particular, when $N=2$ the asymptotic behavior \eqref{eq:asym-u} can be improved. Indeed, the function
$$\hat{u}(x):=u(\frac{x}{|x|^2})-4(\alpha+1)\log|x|$$ is a solution of \eqref{eq_general}$_{N=2}$ in $\mathbb{R}^2 \setminus \{0\}$ which can be continuously extended to the origin as a solution of \eqref{eq_general}$_{N=2}$ in $\mathbb{R}^2$, see for instance \cite[Lemma 2.1 and Remark 2.2]{PT01}, and then
\begin{equation}\label{eq:asym-lim}
	\lim_{|x|\to+\infty}(u(x)+4(\alpha+1)\log|x|)=\hat{u}(0)\,.
\end{equation}
Since $u_r(y)=u(ry)+4(\alpha+1)\log r$ and $-\Delta u_r=\frac{|y|^{2\alpha}e^{u_r}}{r^{2(\alpha+1)}} $ are uniformly bounded in $C_{\mathrm{loc}}(\mathbb{R}^2\setminus \{0\})$ as $r \to +\infty$ in view of \eqref{eq:asym-lim}, by elliptic estimates $u_r$ is uniformly bounded in $C^2_{\mathrm{loc}}(\mathbb{R}^2\setminus \{0\})$ and then
\begin{equation}\label{1853}
|x||\nabla u(x)|+|x|^2 |D^2 u(x)|=|\nabla u_r(\frac{x}{|x|})|+|D^2 u_r(\frac{x}{|x|})|=O(1)
\end{equation}
as $|x|\to \infty$, where $r=|x|$.

Furthermore, we have the following refined result from the discussions in \cite{CGS12} (see Remark 4.3 therein).
\begin{lemma}\label{lem:refined}
	Assume $N=2$ and $\alpha\geq 0$. If $\alpha\notin\mathbb{N}$, then either 
		there exist a constant $A\neq 0$, $k\in\mathbb{N}$ fulfilling $1\leq k<2(\alpha+1)$, and $\theta_0 \in \mathbb{S}^1$, such that
		\begin{equation}\label{eq:asym-u-infty-2}
			u(x)=-4(\alpha+1) \log |x| +\hat{u}(0)+A|x|^{-k} \sin \Big(k \frac{x}{|x|}+\theta_0\Big) + o(|x|^{-k}) \quad
			\text { as } |x| \to +\infty,
		\end{equation} 
or
		
\begin{equation*}
			u(x) =-4(\alpha+1) \log |x|+\hat{u}(0) + O(|x|^{-2(\alpha+1)}) \quad \text { as } |x| \to +\infty.
		\end{equation*}
 If $\alpha \in\mathbb{N}$, then either \eqref{eq:asym-u-infty-2} holds with $1\leq k\leq 2\alpha$ or
		\begin{equation*}
			u(x) =-4(\alpha+1) \log |x|+\hat{u}(0) + O(|x|^{-2\alpha-1})\quad \text { as } |x| \to +\infty.
		\end{equation*}
\end{lemma}

Thanks to Lemma \ref{lem:refined} applied to $\hat u$, we are able to deduce the following result concerning the asymptotics of $\nabla u$ near the origin, which will be used in the next section for deriving the constancy of the $P$-function given by \eqref{eq:def-P}.

\begin{proposition} \label{pro:refine-gradient}
Assume $N=2$ and $\alpha\geq 0$. Then
	\begin{equation}\label{eq:growth-gradient}
		\nabla u(x)= O(|x|^\alpha) \quad\text{ as }|x|\to 0\,.
	\end{equation}
\end{proposition}

\begin{proof}
	First, note that $u\in  C^{2,\beta}(\mathbb{R}^2)\cap C^\infty(\mathbb{R}^2\setminus\{0\})$ when $\alpha\geq 0$, by standard elliptic theory. So \eqref{eq:growth-gradient} clearly holds when $\alpha=0$. In the following, we argue in the case $\alpha>0$.

By virtue of Lemma \ref{lem:refined} applied to $\hat u$, we obtain the following asymptotic behavior of $u$ near the origin. If $\alpha\notin\mathbb{N}$, then either there exist a constant $A\neq 0$, $k\in\mathbb{N}$ fulfilling $1\leq k<2(\alpha+1)$, and $\theta_0 \in \mathbb{S}^1$, such that
	\begin{equation}\label{eq:nNA}
		u(x)= u(0) + A |x|^k \sin\Big(k \frac{x}{|x|}+\theta_0 \Big) + o(|x|^k) \quad
		\text { as } |x| \to 0 \,,
	\end{equation}
or
	\begin{equation}\label{eq:sharp-expansion-u}
		u(x) = u(0) + O(|x|^{2(\alpha+1)}) \quad \text { as } |x| \to 0.
	\end{equation}
If $\alpha \in\mathbb{N}$, then either \eqref{eq:nNA} holds with $1\leq k\leq 2\alpha$ or
	\begin{equation}\label{eq:N-sharp-u}
		u(x)=u(0) + O(|x|^{2\alpha+1})\quad \text { as } |x| \to 0 .
	\end{equation}

We are now going to use basic facts from complex analysis to show that the case \eqref{eq:nNA} above can be excluded when $\alpha\notin\mathbb{N}$, and that $k$ appearing in \eqref{eq:nNA} must be $\alpha+1$ when $\alpha\in\mathbb{N}$. Let $$\tilde{u}(x)=u(x)+2\alpha\log|x|\,,$$
and note that
\begin{equation}\label{eq:u-tilde}
	-\Delta \tilde{u}=e^{\tilde{u}}\quad \text{ in }\mathbb{R}^2\setminus\{0\}.
\end{equation}
In order to rewrite \eqref{eq:u-tilde} in complex notation, we 
let $\bar{z}$ be the conjugate of $z=x+iy\in\mathbb{C}$, and let $\partial_z$, $\partial_{\bar{z}}$ be the first order Wirtinger derivatives given by
\begin{equation*}
	\partial_z =\frac12\left(\partial_{x}-i\partial_{y}\right), \quad 	\partial_{\bar{z}} =\frac12\left(\partial_{x}+i\partial_{y}\right).
\end{equation*} 
 For simplicity, we set $\tilde{u}_z=\partial_z \tilde{u}$, $\tilde{u}_{\bar{z}}=\partial_{\bar{z}} \tilde{u}$ and $\tilde{u}_{z\bar{z}}=\partial_{\bar{z}}\tilde{u}_z$. Then  \eqref{eq:u-tilde} can be rewritten as $$-4\tilde{u}_{z\bar{z}}=e^{\tilde{u}} \quad \text{for }z\neq0 \,,$$
which implies
\begin{equation*}
-4\partial_{\bar{z}}(\tilde{u}_{zz})=-4\partial_{z}(\tilde{u}_{z\bar{z}})=e^{\tilde{u}}\tilde{u}_{z}=-4\tilde{u}_{z\bar{z}}\tilde{u}_{z}=-2\partial_{\bar{z}}(\tilde{u}_{z}^2)\quad \text{for }z\neq0\,,
\end{equation*}
or equivalently
$$\partial_{\bar{z}}\left(\tilde{u}_{zz}-\frac{\tilde{u}_{z}^2}{2}\right)=0 \quad \text{for }z\neq0\,.$$
Thus, $\tilde{u}_{zz}-\frac{\tilde{u}_{z}^2}{2}$ is holomorphic in $\mathbb{C}\setminus\{0\}$ and has a Laurent expansion. Note that 
\begin{equation}\label{eq:tilde-uzz}
	\tilde{u}_{zz}-\frac{\tilde{u}_{z}^2}{2}=u_{zz}-\frac{u_{z}^2}{2}-\alpha u_z z^{-1}-\left(\alpha+\frac{\alpha^2}{2}\right)z^{-2}\,.
\end{equation}
Since $u \in C^2(\mathbb{R}^2)$ and satisfies \eqref{1853} at infinity, we infer from \eqref{eq:tilde-uzz} that
 $$\tilde{u}_{zz}-\frac{\tilde{u}_{z}^2}{2}=O(|z|^{-2}) \quad\text{both as }z \to 0 \text{ and as } |z|\to\infty.$$
This forces
\begin{equation}\label{eq:identity-uzz}
	u_{zz}-\frac{u_{z}^2}{2}-\alpha u_z z^{-1}=0 \quad \text{for }z\neq0\,.
\end{equation}


Assume now that \eqref{eq:nNA} holds. We rewrite it as
 \begin{equation}\label{eq:expansion-u-complex}
 	 u(z)= u(0) + az^k+\bar{a}\bar{z}^k + o(|z|^k) \quad
 	\text {as } z \to 0\,,
 \end{equation}
 where $a\neq 0$ is a complex number. Observe from \eqref{eq:identity-uzz} that $u_z(0)=0$ and then $k\geq 2$ in \eqref{eq:expansion-u-complex}. Since $u \in C^k(\mathbb{R}^2)$ by elliptic estimates, we have that the Taylor polynomial of $u$ at $0$ of order $k$ coincides with \eqref{eq:expansion-u-complex} and then the derivatives of $u$ at $0$ are all determined by \eqref{eq:expansion-u-complex} up to order $k$. Therefore, we derive that
$$u_{zz}-\frac{u_{z}^2}{2}-\alpha u_z z^{-1}=ak(k-1-\alpha)z^{k-2}+o(|z|^{k-2})\quad
 	\text {as } z \to 0 \,.$$
Thus, by \eqref{eq:identity-uzz} it yields that $a=0$ if $\alpha\notin\mathbb{N}$ (which is a contradiction), or that $k=\alpha+1$ if $\alpha\in\mathbb{N}$.

Consequently, we conclude that if $\alpha\notin\mathbb{N}$, $u$ must satisfy \eqref{eq:sharp-expansion-u}, and that if $\alpha\in\mathbb{N}$, either \eqref{eq:nNA} with $k=\alpha+1$ or \eqref{eq:N-sharp-u} does hold. In any case, it holds
$$u(x)-u(0)=O(|x|^{\alpha+1})\quad\text{as }x \to 0\,,$$ which easily implies
$$\nabla u(x)=O(|x|^\alpha) \quad\text{as }x \to 0\,.$$
The proof is completed.
\end{proof}

\section{The $P$-function}\label{sec:P}

This section is devoted to the study of properties of the so-called $P$-function introduced in \eqref{eq:def-P}, which contains three parts. First, in Subsection \ref{sec:P-1} we fix some geometric notation, since we will be working on the Riemannian manifold $(\mathbb{R}^N\setminus\{0\},g)$ where the metric $g$ is defined in \eqref{eq:def-g}. Then we derive in Subsection \ref{sec:P-2} a differential identity satisfied by the $P$-function. At last, in Subsection \ref{sec:P-3} we demonstrate the constancy of the $P$-function under the assumptions of Theorem \ref{thm:2D} or Theorem \ref{thm:HD}.

\subsection{Notation}\label{sec:P-1}
 We use the Einstein summation convention throughout the section. Given an $N$-dimensional Riemannian manifold $(M,g)$, we denote by $(x^1,\cdots,x^N)$ a local coordinate system on $M$ and by $\partial_{x^i}$ $(1\leq i\leq N)$ the corresponding coordinate vector fields tangent to $M$. Let $g_{ij}=g(\partial_ {x^i},\partial_ {x^j})$ and let $g^{ij}$ be such that $g^{ik}g_{kj}=\delta^i_j$, where $\delta^i_j$ denotes the Kronecker symbol. Also, denote by $d\mathscr{V}_g$ the volume form of $(M,g)$.

For any tangent vector $\xi$ to $(M,g)$, we write $\xi=\xi^i\partial_{x^i}$ and define $$|\xi|_g:=\sqrt{g_{ij}\xi^i\xi^j}\,,\quad V_g(\xi):=\frac1N|\xi|^N_g \,,\quad a_g(\xi):=|\xi|_g^{N-2}\xi \,.$$ 
Let $\xi_i=g_{ij}\xi^j$. Then we have
$$\partial_{\xi_i} V_g (\xi)=|\xi|_g^{N-2}\xi^i=a_g^i(\xi)$$
and
\begin{equation}\label{eq:def-Ag}
\partial^2_{\xi_i\xi_j} V_g (\xi) = (N-2) |\xi|_g^{N-4} \xi^i \xi^j + |\xi|_g^{N-2} g^{ij}=:A_g^{ij} (\xi)  \quad\text{ if }\xi\neq0 \,,
\end{equation}
where $A_g^{ij} (\xi):=0$ if $\xi=0$.
It is seen that
\begin{equation}\label{eq:xiA}
	\xi_i A_g^{ij}(\xi) = (N-1) |\xi|^{N-2}_g \xi^j = (N-1) a^j_g(\xi)
\end{equation}
and 
\begin{equation}\label{eq:D3V}
	\xi_j \, \partial^3_{\xi_i \xi_j \xi_k} V_g(\xi) = (N-2)|\xi|_g^{N-4}\left((N-2)\xi^i \xi^k + |\xi|_g^2 \, g^{ik}\right)=(N-2)A_g^{ik}(\xi)
\end{equation}
for $\xi\neq0$.

For a smooth function $v$ on $(M,g)$, set $v_j:=\partial_{x^j}v$ and let $\nabla_g v = v^i\partial_{x^j}$ be the gradient\footnote{When $M$ is a subset of $\mathbb{R}^N$ endowed with the Euclidean metric $\delta_e$, we write $\nabla v$ as usual.} of $v$ with respect to the metric $g$, where $v^i=g^{ij} v_j$. For simplicity, we write $V$ as the function 
\begin{equation}\label{eq:def-V}
	V:=V_g(\nabla_g v)=\frac{1}{N}|\nabla_g v|_g^N
\end{equation}
and $a$ as the tangent vector
\begin{equation}\label{eq:def-a}
	a:=a_g(\nabla_g v)=|\nabla_g v|_g^{N-2}\nabla_g v\,.
\end{equation}
We work with the Levi--Civita connection on $(M,g)$ and let $\nabla_j$ stand for the covariant derivative with respect to the tangent vector $\partial_{x^j}$. Then 
\begin{equation}\label{eq:def-Vj}
	V_j = \partial_{x^j}V=\nabla_j V=\partial_{\xi_k}V_g(\nabla_g v) \, \nabla_j v_k = a^k v_{k,j}
\end{equation}
where $v_{k,j}=\nabla_j\nabla_k \, v$ (i.e. the components of the $g$-Hessian of $v$) and 
\begin{equation}\label{eq:def-ak}
a^k:=a^k_g(\nabla_g v)=|\nabla_g v|_g^{N-2}v^k \,.
\end{equation}
Set
\begin{equation}\label{eq:def-A}
	A^{ij}:= A^{ij}_g(\nabla_g v)\,
\end{equation}
where $A^{ij}_g$ is given by \eqref{eq:def-Ag}. We have
 \begin{equation}\label{eq:def-aik}
 	a^i_{,k} := \nabla_k (a^i)=\nabla_k (\partial_{\xi_i} V_g(\nabla_g v))=\partial^2_{\xi_i \xi_j}V_g(\nabla_g v)\, \nabla_k v_j = A^{ij} v_{j,k} \,.
 \end{equation}

 Let $R_{\ell jik}$ be the components of the Riemann curvature tensor on $(M,g)$. Denote $v_{jk,i}=\nabla_i\nabla_k\nabla_j v$. By the Ricci identity, one has
 \begin{equation}\label{eq:vjki-vjik}
 	v_{jk,i}-v_{ji,k}=v^\ell R_{i \ell jk}\,.
 \end{equation}
In addition, we let $\mathrm{Ric}_g$ be the Ricci curvature tensor on $(M,g)$ and $(\mathrm{Ric}_g)_{ij}$ be its components.

\subsection{A differential identity for the $P$-function} \label{sec:P-2}
Throughout the rest of this section, we adopt the notations in Subsection \ref{sec:P-1} and fix $M=\mathbb{R}^N\setminus\{0\}$ with $g$ the conformal metric defined in \eqref{eq:def-g}, i.e. $$g=|x|^{2\alpha}\delta_e \,,$$ where $\delta_e$ denotes the standard Euclidean metric on $\mathbb{R}^N$. In this case, notice that 
\begin{equation}\label{eq:Ric}
	(\mathrm{Ric}_g)_{ij}(x)=\frac{(N-2)\left(1-(\alpha+1)^2\right)}{|x|^2}\left(\delta^i_j-\frac{x^i x^j}{|x|^2}\right)\quad \text{ for }x\in M
\end{equation}
(see for instance \cite[Formula (2.68)]{CatinoMastrolia}).

In the sequel, we always let $u$ be a solution to problem \eqref{eq_general} and $v$ be the function given by 
\begin{equation}\label{eq:def-v}
	v=e^{-\frac{u}{N}}.
\end{equation}
It can be verified that, in the conformal setting $(M,g)$, $u$ satisfies 
\begin{equation}\label{eq:N-Lap-g-u}
	-\Delta_N^g u :=-\operatorname{div}_g\left(|\nabla_g u|_g^{N-2}\nabla_g u\right)=-|x|^{-N\alpha}\Delta_N u=e^u \quad\text{ in }M,
\end{equation}
where $\operatorname{div}_g$ stands for the divergence operator on $(M,g)$. Likewise, note that
\begin{equation}\label{eq:def-P-N}
	\Delta_N^g v=\operatorname{div}_g (a)=a^i_{,i}=\frac{1}{v} \left[(N-1) |\nabla_g v|_g^N + \frac{1}{N^{N-1}}\right] \quad\text{ in } M,
\end{equation}
where $a$ and $a^i_{,i}$ are introduced in \eqref{eq:def-a} and \eqref{eq:def-aik}, respectively. Here, we point out that both \eqref{eq:N-Lap-g-u} and \eqref{eq:def-P-N} are understood in the weak sense.

From \eqref{eq:def-P-N}, we have that the $P$-function, given by \eqref{eq:def-P}, can be expressed as
\begin{equation}\label{eq:def-P-v2}
	P= \frac{1}{v} \left[(N-1) |\nabla_g v|_g^N + \frac{1}{N^{N-1}}\right] \quad\text{ in } M.
\end{equation}
By Lemma \ref{lem:pre}, it is clear that $v\in C^{1,\beta}_{\mathrm{loc}}(M)$ and $|\nabla v|^{N-2}\nabla v\in W_{\mathrm{loc}}^{1,2}(M)$, thus implying
$P\in  C^{0,\beta}_{\mathrm{loc}}(M)\cap W_{\mathrm{loc}}^{1,2}(M)$ for all $\beta \in (0,1)$. Also, note that $P\in  C^\infty(M\setminus Z)$, since $u\in C^\infty(M\setminus Z)$ (so does $v$), where $Z$ is the critical set of $u$, namely, $$Z:=\{x\in\mathbb{R}^N: \nabla u(x)=0\}.$$ 
Notice that $Z$ coincides with the critical set of $v$ and has zero Lebesgue measure (see for instance \cite[Corollary 1.7]{ACF23}).

We can derive that $P$ satisfies a differential identity pointwise in $M\setminus Z$.  To be specific,  let $W$ be the matrix whose $(i,j)$ entry is given by $a^i_{,j}$ (see \eqref{eq:def-aik}) and 
\begin{equation}\label{eq:def-E}
	E=W-\frac{P}{N}\mathrm{Id}\,,
\end{equation}
where $\mathrm{Id}$ denotes the identity matrix. By \eqref{eq:def-P-N}--\eqref{eq:def-P-v2},  note that
\begin{equation} \label{1205}
\mathrm{Tr} (E)=\mathrm{Tr} (W)-P=0, \quad \mathrm{Tr} (E^2)=\mathrm{Tr} (W^2)-\frac{[\mathrm{Tr} (W)]^2}{N}.
\end{equation}
We have the following result.
\begin{lemma}\label{lem:identity-P}
	Assume $N\geq 2$ and $\alpha>-1$. Then there holds
	\begin{equation}\label{eq:di-Pj}
		\nabla_i (v^{2-N} A^{ij} P_j) = N (N-1) v^{1-N} \left[\mathrm{Tr} (E^2)+(\mathrm{Ric}_g)_{ij}a^ia^j\right] \quad\text{ in } M\setminus Z \,,
	\end{equation}
	where $\rm{Tr}(E^2)$ denotes the trace of $E^2$, $A^{ij}$ and $a^i$ are introduced in \eqref{eq:def-A} and \eqref{eq:def-ak}, respectively.
\end{lemma}

\begin{proof}
	Note that
	\begin{equation}\label{eq:Di-vAP}
		\nabla_i (v^{2-N} A^{ij} P_j) = v^{1-N} \left[ v A^{ij}P_{i,j} + v A^{ij}_{,i} P_j - (N-2) v_i A^{ij}P_j  \right].
	\end{equation}
	We are going to calculate each term on the right hand side of \eqref{eq:Di-vAP}. First, in view of \eqref{eq:def-V} and \eqref{eq:def-P-v2}, we have
	\begin{equation} \label{eq:P_j}
		P_j =\nabla_j P =- \frac{P}{v} v_j + \frac{N(N-1)}{v} V_j \,,
	\end{equation}
so that
	\begin{align*}
	P_{j,i} =\nabla_i P_j & = -\frac{P_i v_j}{v} + \frac{P}{v^2}v_i v_j - \frac{P}{v}v_{j,i} - \frac{N(N-1)}{v^2}v_i V_j + \frac{N(N-1)}{v} V_{i,j} \\
		& = -\frac{P_i v_j}{v}   -\frac{P_j v_i}{v} - \frac{P}{v}v_{i,j}  + \frac{N(N-1)}{v} V_{j,i}  \,.
	\end{align*}
	Thus, by \eqref{eq:def-aik} and \eqref{eq:def-P-N}--\eqref{eq:def-P-v2}, we get
	\begin{align}
		v A^{ij}P_{i,j}  & = v A^{ij} \left( -\frac{P_i v_j}{v}   -\frac{P_j v_i}{v} - \frac{P}{v}v_{i,j}  + \frac{N(N-1)}{v} V_{i,j}  \right) \notag\\
		& = -  2 A^{ij} P_i v_j - P^2 + N(N-1) A^{ij} V_{i,j}  \,. \label{eq:vAPij0}
	\end{align}
Furthermore, by \eqref{eq:def-Ag}, \eqref{eq:def-Vj}--\eqref{eq:vjki-vjik}, \eqref{eq:def-P-N}--\eqref{eq:def-P-v2} and \eqref{eq:P_j}, we derive 
	\begin{align}
		A^{ij} V_{i,j}  & = A^{ij} \nabla_i (a^k v_{k,j}) \notag\\
		& = A^{ij} a^k_{,i}  \, v_{k,j}  + A^{ij} a^k v_{jk,i} \notag\\
		& = a^i_{,k} \, a^k_{,i} + A^{ij} a^k v_{jk,i} \notag\\
		&= a^i_{,k} \, a^k_{,i}  + A^{ij} a^k (v_{ji,k}+v^\ell R_{i \ell jk}) \notag\\
		& = a^i_{,j} \, a^j_{,i} +  a^k  \nabla_k (A^{ij}v_{j,i}) - A^{ij}_{,k} a^k v_{j,i}+A^{ij} a^k v^\ell R_{i \ell jk} \notag\\
		& = a^i_{,j} \, a^j_{,i} +  a^k  \nabla_k (A^{ij}v_{j,i} ) - \left(\partial^3_{\xi_i \xi_j \xi_\ell} V_g(\nabla_g v)\right) v_{\ell, k}\,  a^k v_{j,i} + A^{ij} a^k v^\ell R_{i\ell jk} \notag\\
		& =a^i_{,j}\, a^j_{,i} + a^k P_k - \left(\partial^3_{\xi_i \xi_j \xi_\ell} V_g(\nabla_g v)\right) V_\ell \, v_{j,i} +A^{ij} a^k v^\ell  R_{i\ell jk} \,. \label{eq:AijVij}
	\end{align}
and
 \begin{equation} \label{eq:vAiPj}
	v A^{ij}_{,i} P_j = ( -P v_j + N(N-1) V_j) \left(\partial^3_{\xi_i \xi_j \xi_l} V_g(\nabla_g v)\right)v_{l,i} \,. 
\end{equation}
Hence, we conclude from \eqref{eq:Di-vAP} and \eqref{eq:vAPij0}--\eqref{eq:vAiPj} that
	\begin{align}
		&\nabla_i (v^{2-N}A^{ij} P_j) \notag\\
		 = \,& v^{1-N} \Big\{ -  2 A^{ij} P_i v_j - P^2 \notag\\
		 & + N(N-1) \left[a^i_{,j}\, a^j_{,i} + a^k P_k - \left(\partial^3_{\xi_i \xi_j \xi_\ell} V_g(\nabla_g v)\right) V_\ell \, v_{j,i} +A^{ij} a^k v^\ell R_{i\ell jk}\right] \notag\\
		& +  (-P v_j + N(N-1) V_j) \left(\partial^3_{\xi_i \xi_j \xi_l} V_g(\nabla_g v)\right)v_{l,i}  - (N-2) v_i A^{ij}P_j  \Big\} \notag\\
		 = \, & v^{1-N} \Big[-P^2+ N(N-1) \left(a^i_{,j}\, a^j_{,i} +  a^k  P_k + A^{ij} a^k v^\ell R_{i\ell jk} \right) \notag\\
		 & -P v_j \left(\partial^3_{\xi_i \xi_j \xi_l} V_g(\nabla_g v)\right)v_{l,i}  -N v_i A^{ij}P_j  \Big]. \label{eq:Di-vAP-2}
	\end{align}
Note from \eqref{eq:xiA} and \eqref{eq:def-ak}--\eqref{eq:def-A} that
	$$
	v_i A^{ij}P_j = (N-1) a^j P_j \,.
	$$
	Then \eqref{eq:Di-vAP-2} becomes
	\begin{align}
		\nabla_i (v^{2-N}A^{ij} P_j) = \,& v^{1-N} \Big[-P^2+ N(N-1) \left(a^i_{,j}\, a^j_{,i} + A^{ij} a^k v^\ell R_{i\ell jk} \right) \notag\\
		&  -P v_j \left(\partial^3_{\xi_i \xi_j \xi_l} V_g(\nabla_g v)\right)v_{l,i} \Big]. \label{eq:vAPj}
	\end{align}
	By virtue of \eqref{eq:D3V}, \eqref{eq:def-A}--\eqref{eq:def-aik} and \eqref{eq:def-P-N}--\eqref{eq:def-P-v2}, we find
	\begin{equation}\label{1044}
	v_j \left(\partial^3_{\xi_i \xi_j \xi_l} V_g(\nabla_g v)\right) v_{l,i}= (N-2) A^{il} v_{l,i} = (N-2) P \,.
\end{equation}
	Consequently, by \eqref{eq:def-Ag}, \eqref{eq:def-ak}--\eqref{eq:def-A},  \eqref{1205} and \eqref{eq:vAPj}-\eqref{1044}, we arrive at
	\begin{align*}
		\nabla_i (v^{2-N}A^{ij} P_j) &=N (N-1) v^{1-N} \left(a^i_{,j}\, a^j_{,i} - \frac{P^2}{N} + A^{ij} a^k v^\ell R_{i\ell jk} \right) \\
		&= N (N-1) v^{1-N} \left[\mathrm{Tr} (E^2)+ (\mathrm{Ric}_g)_{lk} a^l a^k \right]
	\end{align*}
thanks to the symmetries of the Riemann tensor. This completes the proof.
\end{proof}

\begin{remark}\label{rk:Tr-E}
	Note that in $M\setminus Z$, $\mathrm{Tr}(E^2)\geq 0$, and $\mathrm{Tr}(E^2)=0$ if and only if $E=0$. Indeed, in view of \eqref{eq:def-Ag} and \eqref{eq:def-A}--\eqref{eq:def-aik}, the matrix $W=(a^i_{,j})$ can be written as
	 \begin{equation}\label{eq:rewrite-W}
	 	W=|x|^{-N\alpha}\mathcal{A}(v)H_g(v),
	 \end{equation}
	 where $H_g(v)$ denotes the $g$-Hessian of $v$ whose components are $v_{i,j}$ and 
	 \begin{equation}\label{eq:def-B}
	 	\mathcal{A}(v)=|\nabla v|^{N-2}\mathrm{Id}+(N-2)|\nabla v|^{N-4}\nabla v\otimes\nabla v \,.
	 \end{equation}
	  Thus, as argued in the proof of \cite[Lemma 6.3]{SZ02}, by using \eqref{1205}, it is easy to see that
	 $$\mathrm{Tr}(E^2)=\mathrm{Tr}(W^2)-\frac{[\mathrm{Tr}(W)]^2}{N}\geq 0$$
	and the equality holds if and only if $W=\frac{\rm{Tr}(W)}{N}\mathrm{Id}$, i.e. $E=0$.
\end{remark}

\begin{remark}\label{rk:P2D}
	As $u\in C^\infty(M)$ when $N=2$, it is seen that in this case \eqref{eq:di-Pj} holds globally in $M$ and, by \eqref{eq:def-Ag}, \eqref{eq:def-aik}, \eqref{eq:Ric}, \eqref{1205}, it can be rewritten as 
	\begin{equation} \label{eq:Delta-g-P}
		\Delta_g P = 2 v^{-1} \left| H_g(v) - \frac{P}{2} g \right|_g^2 \quad\text{ in } M,
	\end{equation}
	where $\Delta_g$ is the Laplace--Beltrami operator with respect to the metric $g$, and hence $P$ is subharmonic in $M$.
\end{remark}

	 From Lemma \ref{lem:identity-P}, we can derive the following inequality.
	 
\begin{lemma}\label{lem:DvAP}
	Let $N$ and $\alpha$ be as in Theorem \ref{thm:2D} or Theorem \ref{thm:HD}. Then for any $t\in\mathbb{R}$,
	\begin{equation}\label{eq:DvAPt}
		\nabla_i (v^{2-N} A^{ij}  P^{t-1} P_j) \geq \left(t-\frac1N\right) P^{t-2} v^{2-N} A^{ij}  P_i P_j  \quad\text{ in }M\setminus Z \,.
	\end{equation}
\end{lemma}

\begin{proof}	
	Notice from \eqref{eq:Ric} that $\mathrm{Ric}_g\geq0$ under the assumptions of the lemma. Thus, it follows from \eqref{eq:di-Pj} that
\begin{equation}\label{1453}
\nabla_i (v^{2-N} A^{ij}  P^{t-1} P_j)\geq N(N-1)P^{t-1} v^{1-N}\mathrm{Tr} (E^2) + (t-1) P^{t-2} v^{2-N} A^{ij}  P_i P_j \,.
\end{equation}	
In view of \eqref{eq:xiA}, \eqref{eq:def-Vj}--\eqref{eq:def-aik}, \eqref{eq:def-E} and \eqref{eq:P_j}, we derive
	\begin{align}
		A^{ij}P_iP_j &=\frac{N^2(N-1)|\nabla_g v|_g^{N-2}}{v^2}\left((W^2)^j_{i}v^iv_j -\frac{2P}{N} W^j_{i} v^i v_j+ \frac{P^2}{N^2} |\nabla_g v|_g^2\right) \notag\\
		 &= \frac{N^2(N-1)|\nabla_g v|_g^{N-2}}{v^2} \left((E^2)^j_{i} v^iv_j\right) \label{eq:AijPiPj}
	\end{align}
where $W^j_{i}$, $(W^2)^j_{i}$ and $(E^2)^j_{i}$ denote the $(j, i)$ entry of the matrices $W$, $W^2$ and $E^2$, respectively. 

Since the matrix $\mathcal{A}(v)$ given by \eqref{eq:def-B} is symmetric with positive eigenvalues, one can rewrite $E$ defined in \eqref{eq:def-E} as $$E=|x|^{-N\alpha}\mathcal{A}(v)F,$$
where $F=H_g(v)-\frac{P}{N}|x|^{N\alpha}\mathcal{A}^{-1}(v)$, with $H_g(v)$ as in \eqref{eq:rewrite-W}. Then, as argued in \cite{Ou22} for Lemma 2.7 therein, we infer that
 $$(E^2)^j_i v^iv_j\leq \mathrm{Tr} (E^2)|\nabla_g v|_g^2\,.$$
Hence, we conclude from \eqref{eq:AijPiPj} that 
$$A^{ij}P_iP_j \leq N^2(N-1) v^{-2} |\nabla_g v|_g^{N} \mathrm{Tr} (E^2) $$
and then
	\begin{align*}
		P^{t-2} v^{2-N} A^{ij}  P_i P_j  & \leq  N^2(N-1) P^{t-2} v^{-N} |\nabla_g v|_g^{N} \mathrm{Tr} (E^2) \\
		 & \leq N^2P^{t-1} v^{1-N}\mathrm{Tr} (E^2)
	\end{align*}
thanks to $(N-1)|\nabla_g v|_g^{N}\leq Pv$ in view of \eqref{eq:def-P-v2},  which, inserted into \eqref{1453},  establishes the validity of \eqref{eq:DvAPt}.
\end{proof}

\subsection{Constancy of the $P$-function}\label{sec:P-3}
By virtue of Lemma \ref{lem:DvAP}, we find that the $P$-function we are considering is actually constant; see Proposition \ref{pro:P0} below. To this end, let us first translate inequality \eqref{eq:DvAPt} on $M\setminus Z$ into the following integral version on $M$.
\begin{lemma}\label{lem:intergal-idenitity-Pt} 
	Let $N$ and $\alpha$ be as in Theorem \ref{thm:2D} or Theorem \ref{thm:HD}. Then for any $t\in\mathbb{R}$ and $\varphi\in C_c^\infty(\mathbb{R}^N\setminus\{0\})$ with $\varphi\ge0$,
	\begin{equation}\label{eq:intergal-idenitity-Pt}
		-\int_{\mathbb{R}^N\setminus\{0\}} v^{2-N} A^{ij} P^{t-1} P_j \varphi_i \, d \mathscr{V}_g\geq \left(t-\frac1N\right) \int_{\mathbb{R}^N\setminus\{0\}} P^{t-2} v^{2-N} A^{ij}  P_i P_j \varphi \, d\mathscr{V}_g \,.
	\end{equation}
\end{lemma}

\begin{proof}
	First, recall that $v\in C^{1,\beta}_{\mathrm{loc}}(\mathbb{R}^N\setminus \{0\})$ and
	$P\in  C^{0,\beta}_{\mathrm{loc}}(\mathbb{R}^N\setminus\{0\})\cap W_{\mathrm{loc}}^{1,2}(\mathbb{R}^N\setminus\{0\})$. For $0<\varepsilon\ll1$, we let 
$$
\psi_{\varepsilon}(x):=\min\left\{\frac{\left(|\nabla v(x)|^{N-1}-\varepsilon^{N-1}\right)^+}{\varepsilon^{N-1}}\, ,\,1\right\}.
$$ 
It follows that $\psi_\varepsilon(x)\in W_{\mathrm{loc}}^{1,2}(\mathbb{R}^N\setminus\{0\})\cap C(\mathbb{R}^N\setminus\{0\})$. By setting 
\begin{equation*}
	U^\varepsilon:=\{x\in\mathbb{R}^N:|\nabla v(x)|\geq2^{\frac{1}{N-1}}\varepsilon\}\, ,
\end{equation*}
we have
\begin{equation*}
	\psi_{\varepsilon}(x)=
	\begin{cases}
		1 &  \text{if }x\in U^\varepsilon \\
		0 & \text{if }x\in Z \, ,
	\end{cases}
\end{equation*}
and $\psi_\varepsilon(x)\to\psi_0(x)$ as $\varepsilon\to0$ for every $x\in\mathbb{R}^N\setminus\{0\}$, where 
\begin{equation*}
	\psi_0(x):=
	\begin{cases}
		1 &  \text{if }x\in \mathbb{R}^N\setminus Z\\
		0 & \text{if }x\in Z \,.
	\end{cases}
\end{equation*}
Given $\varphi\in C_c^\infty(\mathbb{R}^N\setminus\{0\})$, let us write it as
\begin{equation*}
	\varphi=\psi_\varepsilon\varphi+(1-\psi_\varepsilon)\varphi \,.
\end{equation*}
Observe that $\psi_\varepsilon\varphi$ is compactly supported in $\mathbb{R}^N\setminus (Z\cup\{0\})$ and it belongs to $ W_{0}^{1,2}(\mathbb{R}^N\setminus (Z\cup\{0\}))$. Thanks to Lemma \ref{lem:DvAP}, it is clear that \eqref{eq:intergal-idenitity-Pt} holds for $\varphi\in C_c^\infty(\mathbb{R}^N\setminus (Z\cup\{0\}))$ with $\varphi\geq 0$. Hence, by density, one can conclude that \eqref{eq:intergal-idenitity-Pt} also holds for $\varphi$ replaced by $\psi_\varepsilon\varphi$.
	
	On the other hand, since 
$$v\geq \delta, \quad |\nabla P|\leq C|\nabla v|\Big[1+|\nabla (|\nabla v|^{N-1})|\Big] \hbox{ on }\mathrm{supp} (\varphi)$$ 
in view of \eqref{eq:def-v} and \eqref{eq:def-P-v2} for some $\delta,C>0$, by \eqref{eq:def-Ag} and \eqref{eq:def-A} we can derive that
\begin{align*}
		& \lim_{\varepsilon \to 0} \left|\int_{\mathbb{R}^N\setminus\{0\}} v^{2-N} A^{ij}  P^{t-1} P_j ((1-\psi_\varepsilon)\varphi)_i \, d\mathscr{V}_g \right| \\
		 \lesssim &  \lim_{\varepsilon \to 0} \int_{\mathrm{supp} (\varphi) \setminus U^\varepsilon}|\nabla v|^{N-2}|\nabla P|\left(|\nabla\varphi|+|\varphi||\nabla\psi_\varepsilon|\right) dx  \\
		\lesssim &  \lim_{\varepsilon \to 0} \int_{\mathrm{supp} (\varphi) \setminus U^\varepsilon} |\nabla v|^{N-1}  \left(1+|\nabla(|\nabla v|^{N-1})|\right)
 \left(|\nabla\varphi|+\varepsilon^{1-N}|\varphi||\nabla(|\nabla v|^{N-1})|\right) dx \\
\lesssim   &  \int_Z \left(1+|\nabla(|\nabla v|^{N-1})|\right)|\nabla(|\nabla v|^{N-1})| dx=0
\end{align*}
in view of $|Z|=0$ and 
\begin{align*}
		&  \lim_{\varepsilon \to 0} \left|\int_{\mathbb{R}^N\setminus\{0\}} P^{t-2} v^{2-N} A^{ij}  P_i P_j (1-\psi_{\varepsilon})\varphi \, d\mathscr{V}_g \right|\\
\lesssim &  \lim_{\varepsilon \to 0} \int_{\mathrm{supp} (\varphi) \setminus U^\varepsilon} |\nabla v|^N  \left(1+|\nabla(|\nabla v|^{N-1})|\right)^2 (1-\psi_{\varepsilon}) dx=0 \,.
	\end{align*}

	Thus, we conclude that 
	\begin{align*}
		& -\int_{\mathbb{R}^N\setminus\{0\}} v^{2-N} A^{ij}  P^{t-1} P_j \varphi_i \, d\mathscr{V}_g \\
		 = & -\int_{\mathbb{R}^N\setminus\{0\}} v^{2-N} A^{ij}  P^{t-1} P_j (\psi_{\varepsilon}\varphi)_i \, d\mathscr{V}_g \\ &-\int_{\mathbb{R}^N\setminus\{0\}} v^{2-N} A^{ij}  P^{t-1} P_j ((1-\psi_{\varepsilon})\varphi)_i \, d\mathscr{V}_g \\
		\geq & \left(t-\frac1N\right) \int_{\mathbb{R}^N\setminus\{0\}} P^{t-2} v^{2-N} A^{ij}  P_i P_j \psi_{\varepsilon}\varphi \, d\mathscr{V}_g \\
		& -\int_{\mathbb{R}^N\setminus\{0\}} v^{2-N} A^{ij}  P^{t-1} P_j ((1-\psi_{\varepsilon})\varphi)_i \, d\mathscr{V}_g \,,
	\end{align*} 
	which yields \eqref{eq:intergal-idenitity-Pt} by letting $\varepsilon\to0$. This completes the proof.
\end{proof}

In terms of Lemma \ref{lem:intergal-idenitity-Pt}, we obtain the following rigidity result.
\begin{proposition}\label{pro:P0}
	Let $N$ and $\alpha$ be as in Theorem \ref{thm:2D} or Theorem \ref{thm:HD}. Then $P\equiv P_0$ in $\mathbb{R}^N\setminus\{0\}$ for some constant $P_0$. 
\end{proposition}

\begin{proof}
		Denote by $B^g_s$ the geodesic ball (with respect to the metric $g$) centered at the origin of radius $s$. Given $0<r<R$, let $\varphi\geq 0$ be such that
	\begin{equation*}
		\begin{cases}
			\varphi=1 & \text { in } B_R^g \backslash B_{2 r}^g \\
			\varphi=0 & \text { in } B_r^g \cup (\mathbb{R}^N \backslash B_{2 R}^g) \\ 
			\left|\nabla_g \varphi\right|_g \leq \frac{C}{R} & \text { in } B_{2 R}^g \backslash B_R^g \\
			\left|\nabla_g \varphi\right|_g \leq \frac{C}{r} & \text { in } B_{2 r}^g \backslash B_r^g \,,
		\end{cases}
	\end{equation*}
where $C$ is some constant independent of $r$ and of $R$.

By applying Lemma \ref{lem:intergal-idenitity-Pt} with $\varphi^2$ in place of $\varphi$ as the test function and by the Cauchy--Schwartz inequality, we have
	\begin{align}
		&\quad \left(t-\frac1N\right)\int_{\mathbb{R}^N\setminus\{0\}} P^{t-2} v^{2-N} A^{ij} P_i P_j \varphi^2\, d\mathscr{V}_g \notag \\
		& \leq -2 \int_{\mathbb{R}^N\setminus\{0\}} P^{t-1}v^{2-N} A^{ij} P_j \varphi\varphi_i \, d\mathscr{V}_g \notag\\
	& \leq 2 \int_{\mathbb{R}^N\setminus\{0\}}\left[ P^{\frac{t-2}{2}}v^{\frac{2-N}{2}}\varphi (A^{ij}P_i P_j)^{\frac12}\right]\left[ P^{\frac{t}{2}}v^{\frac{2-N}{2}}(A^{ij}\varphi_i \varphi_j)^{\frac12}\right] d\mathscr{V}_g \notag\\
		& \leq 2 \left[\int_{\mathbb{R}^N\setminus\{0\}} P^{t-2}v^{2-N} \varphi^2 A^{ij}  P_i P_j \, d\mathscr{V}_g \right]^{1/2}\left[\int_{\mathbb{R}^N\setminus\{0\}} P^t v^{2-N} A^{ij}\varphi_i \varphi_j \, d\mathscr{V}_g\right]^{1 / 2}. \label{eq:pf-prop-P_0-1}
		\end{align}
Thus, for $t>\frac{1}{N}$,
	\begin{align}
		& \quad \left(t-\frac1N\right)^2\int_{B_R^g\setminus B_{2r}^g} P^{t-2}v^{2-N} A^{ij} P_i P_j \, d\mathscr{V}_g \notag\\
		& \leq \left(t-\frac1N\right)^2\int_{\mathbb{R}^N\setminus\{0\}} P^{t-2} v^{2-N} A^{ij} P_i P_j \varphi^2\, d\mathscr{V}_g \notag\\
		&\leq 4 \int_{\mathbb{R}^N\setminus\{0\}} P^t v^{2-N} A^{ij}\varphi_i \varphi_j \, d\mathscr{V}_g \label{eq:Pt-v-A-phi}\\
		& \leq 4(N-1)\int_{(B_{2 R}^g \backslash B_R^g)\cup (B_{2r}^g \backslash B_r^g)} P^t v^{2-N}|\nabla_g v|_g^{N-2} |\nabla_g\varphi|_g^2 \, d\mathscr{V}_g \notag\\
		& \lesssim \frac{1}{R^2} \int_{B_{2 R}^g \backslash B_R^g} P^t v^{2-N} |\nabla_g v|_g^{N-2} \, d\mathscr{V}_g + \frac{1}{r^2} \int_{B_{2r}^g \backslash B_r^g} P^t v^{2-N} |\nabla_g v|_g^{N-2}\, d\mathscr{V}_g \,. \label{eq:integrable-vAP}
	\end{align}
	
	We are going to show that, for a suitable $t>\frac{1}{N}$, the sum of the two integrals in \eqref{eq:integrable-vAP} is bounded, uniformly as $R\to+\infty$ and $r\to 0$.

	Indeed, by the asymptotics of $u$ at infinity (see \eqref{eq:asym-u} and \eqref{eq:asym-grad}), we find that the first integral in \eqref{eq:integrable-vAP} is uniformly bounded with respect to $R$.
	
	  By Lemma \ref{lem:pre}, $v\in C^1_{\mathrm{loc}}(\mathbb{R}^N)$ when $\alpha>-\frac{1}{N}$. Then it is easy to verify that when $-\frac{1}{N}<\alpha\leq 0$, the second integral in \eqref{eq:integrable-vAP} is bounded uniformly as $r\to0$. This is also true when $N=2$ and $\alpha>0$, thanks to Proposition \ref{pro:refine-gradient}. 
	  
	  If $-1<\alpha\leq-\frac{1}{N}$, then by H\"older inequality and property \eqref{eq:KM-gradient} we have 
	   \begin{equation}\label{eq:pre-alpha-1N}
	  	\int_{B_s}|x|^{-\alpha(Nt-2)}|\nabla v|^{Nt+N-2} \, dx \lesssim s^{-\alpha(Nt-2)+N-\frac{(N-p)(Nt+N-2)}{(N-1)p}}\,,
	  \end{equation}
	 for any $1<p<-\frac{1}{\alpha}\leq N$ so that $Np(N-1)>(N-p)(Nt+N-2)$. Take $t>\frac{1}{N}$ and $p$ such that $$-\alpha(Nt-2)+N-\frac{(N-p)(Nt+N-2)}{(N-1)p}\geq 2(\alpha+1)\,.$$
For instance, one can choose $p=-\frac{1}{\alpha}-\delta$ and $t=\dfrac{1}{N} + \delta$ for $\delta$ small enough.
	Then \eqref{eq:pre-alpha-1N} implies
	\begin{equation*}
		\int_{B_s}|x|^{N\alpha-\alpha(Nt+N-2)}|\nabla v|^{Nt+N-2} \, dx =O(s^{2(\alpha+1)})\quad\text{as }s\to 0\,.
	\end{equation*}
This entails the estimate
	 \begin{equation*}
	 	\frac{1}{r^2} \int_{B_{r}^g} P^t v^{2-N} |\nabla_g v|_g^{N-2}\, d\mathscr{V}_g= O(1) \quad\text{as }r\to 0\,.
	 \end{equation*}

As a result, we have proved that for suitable $t>\frac{1}{N}$, \eqref{eq:integrable-vAP} is uniformly bounded with respect to $R$ and $r$. Hence, letting $R\to+\infty$ and $r\to 0$ in \eqref{eq:integrable-vAP} gives
\begin{equation}\label{eq:integrability}
	\int_{\mathbb{R}^N\setminus\{0\}} P^{t-2}v^{2-N} A^{ij} P_i P_j \, d\mathscr{V}_g <+\infty\,.
\end{equation}
Now, going back to \eqref{eq:pf-prop-P_0-1} and using \eqref{eq:Pt-v-A-phi}, we obtain
	\begin{align*}
			& \left(t-\frac1N\right)\int_{B_R^g\setminus B_{2r}^g} P^{t-2}v^{2-N} A^{ij} P_i P_j \, d\mathscr{V}_g \\
			 \leq & \left(t-\frac1N\right)\int_{\mathbb{R}^N\setminus\{0\}} P^{t-2} v^{2-N} A^{ij} P_i P_j \varphi^2\, d\mathscr{V}_g \\
			\lesssim & \left(\int_{(B_{2 R}^g \backslash B_R^g)\cup (B_{2r}^g \backslash B_r^g)} P^{t-2} v^{2-N} A^{ij}  P_i P_j \, d\mathscr{V}_g \right)^{1/2}.
	\end{align*}
 Letting $R\to+\infty$ and $r\to 0$ in the above inequality and taking advantage of \eqref{eq:integrability} yields
$$\int_{\mathbb{R}^N\setminus\{0\}} P^{t-2}v^{2-N} A^{ij}  P_i P_j \, d\mathscr{V}_g =0 \,.$$
This implies $A^{ij}P_i P_j = 0$ in $\mathbb{R}^N \setminus \{0\}$ and then $\nabla P=0$ thanks to the ellipticity of $A^{ij}$. Thus $P\equiv P_0$ for some constant $P_0$, completing the proof.
\end{proof}

\section{Proofs of Theorems \ref{thm:2D} and \ref{thm:HD}} \label{sec:pf}
Thanks to Proposition \ref{pro:P0} and Lemma \ref{lem:identity-P}, we prove Theorems \ref{thm:2D} and \ref{thm:HD} in this section.

\begin{proof}[Proof of Theorem \ref{thm:2D}]
	Since $P \equiv P_0$ by Proposition {\red \ref{pro:P0}}, it follows from \eqref{eq:Delta-g-P} in Remark \ref{rk:P2D} that
	\begin{equation}\label{eq:kg0}
		H_g(v) -\frac{P_0}{2} g \equiv 0 \quad\text{ in }\mathbb{R}^2\setminus\{0\}\,,
	\end{equation}
where $v$ is given by \eqref{eq:def-v}. Note that for $g=e^{2 \phi} \delta_e$ with $\phi=\alpha \log |x|$,
	\begin{equation}\label{eq:def-Hg}
		H_g(v) =D^2 v+(\nabla \phi \cdot \nabla v) \delta_e- (d \phi \otimes d v+d v \otimes d \phi)\,,
	\end{equation}
where, here and in the sequel, $D^2v=(v_{ij})$ and the bracket $(\nabla \phi \cdot \nabla v)$ stand for the Hessian of $v$ and the inner product in the Euclidean space $\mathbb{R}^N$ with the metric $\delta_e$, respectively. Combining \eqref{eq:kg0}--\eqref{eq:def-Hg}, we get
	\begin{equation*}
		v_{12}=\frac{\alpha}{|x|^2}(x_1v_2+x_2v_1) \quad\text{ in }\mathbb{R}^2\setminus\{0\}. 
	\end{equation*}
In complex notation, this can be rewritten as
	\begin{equation}\label{eq:com-v12}
	\tilde{v}_{zz}-\tilde{v}_{\bar{z}\bar{z}}=\frac{\alpha}{|z|^2}(\tilde{v}_{z}\bar{z}-\tilde{v}_{\bar{z}}z) \quad\text{ for }z\neq 0\,.
\end{equation}

Recall from \eqref{eq:def-P-N}--\eqref{eq:def-P-v2} that
 \begin{equation*}
 	\Delta v = |x|^{2\alpha}\Delta_g v = P_0|x|^{2\alpha}\quad\text{ in }\mathbb{R}^2\setminus\{0\}. 
 \end{equation*}
 Let $\tilde{v}=v-\frac{P_0}{4(\alpha+1)^2}|x|^{2\alpha+2}$. It is seen that $\tilde v$ is continuous at the origin and satisfies
	$$\Delta \tilde{v}=0 \quad\text{in }\mathbb{R}^2\setminus\{0\}.$$
Thus, $\tilde v$ can be expanded in complex notation as
	\begin{equation}\label{eq:series-v}
		\tilde v=\sum_{k=0}^\infty (b_k z^k+ \overline{b_k}\bar{z}^k)\,,
	\end{equation}
	where $k\in\mathbb{N}$ and $b_k, z\in\mathbb{C}$.  Substituting \eqref{eq:series-v} into \eqref{eq:com-v12} gives
	$$\sum_{k=1}^\infty k(k-1)(b_k z^{k-2}-\overline{b_k}\bar{z}^{k-2})=\sum_{k=1}^\infty \alpha k(b_k z^{k-2}-\overline{b_k}\bar{z}^{k-2})\,.$$
	Thus, we deduce that  if $\alpha\in\mathbb{N}$, then $b_k=0$ for $k\notin\{0, \alpha+1\}$, and that if $\alpha\notin\mathbb{N}$, then $b_k=0$ for any $k\ge 1$. This leads to
	$$\tilde v=
	\begin{cases}
		b_0+\overline{b_0} & \text{if }\alpha\notin\mathbb{N}\\
		b_0+\overline{b_0}+ b_{\alpha+1} z^{\alpha+1} + \overline{b_{\alpha+1}}\bar{z}^{\alpha+1}& \text{if } \alpha\in\mathbb{N}\,.
	\end{cases}
	$$
	Hence, we conclude that
	\begin{equation}\label{eq:formula-v}
		v=\tilde{v}+\frac{P_0}{4(\alpha+1)^2}|x|^{2\alpha+2} 
		=\frac{P_0}{4(\alpha+1)^2}|z^{\alpha+1}+c|^2+b
	\end{equation}
	for some $b\in\mathbb{R}$ and $c\in\mathbb{C}$, with $c=0$ when $\alpha\notin\mathbb{N}$. Furthermore, since \eqref{eq_general} with $N=2$ reads in complex notation as
	\begin{equation}\label{eq:u-complex}
		-4u_{z\bar{z}}=|z|^{2\alpha}e^u \,,
	\end{equation}
	we can determine $b=\frac{1}{2P_0}$ by substituting $u=-2\log v$, with the form \eqref{eq:formula-v}, into \eqref{eq:u-complex}.  Consequently, from \eqref{eq:formula-v} with $b=\frac{1}{2P_0}$ we see that $u$ takes the form \eqref{eq:thm-u}, with $\lambda=\frac{P_0}{\sqrt{2}(\alpha+1)}$. This completes the proof.
\end{proof}

\begin{proof}[Proof of Theorem \ref{thm:HD}]
	Note from \eqref{eq:Ric} that, when $-1<\alpha\leq 0$, 
	\begin{equation}\label{eq:nonnegative-Ric}
		(\mathrm{Ric}_g)_{ij}a^ia^j=\frac{(N-2)\left(1-(\alpha+1)^2\right)}{|x|^2}\left[|a|^2-\frac{(a\cdot x)^2}{|x|^2}\right]\geq 0 \,.
	\end{equation}
	Recall that here $a$ is given by \eqref{eq:def-a} and its components $a^i$ are introduced in \eqref{eq:def-ak}.
	
	As in \eqref{eq:def-v}, let $v=e^{-\frac{u}{N}}$. We shall show that $P\equiv P_0$ implies
	\begin{equation}\label{eq:fmu-v}
		v(x)=\frac{P_0^{\frac{1}{N-1}}(N-1)}{[N(\alpha+1)]^\frac{N}{N-1}}|x-x_0|^{\frac{N(\alpha+1)}{N-1}}+\frac{1}{P_0 N^{N-1}} \quad\text{ in }\mathbb{R}^N \,,
	\end{equation}
for some $x_0\in\mathbb{R}^N$, with $x_0=0$ when $\alpha\neq 0$.
	
	Indeed, in view of \eqref{eq:di-Pj}, it is forced that 
	\begin{equation}\label{eq:TrE=Ric=0}
		\mathrm{Tr}(E^2)=0\quad\text{ and }\quad (\mathrm{Ric}_g)_{ij}a^ia^j=0\,, \quad\text{ in }\mathbb{R}^N\setminus (Z\cup\{0\})\,.
	\end{equation}
	Then we get $E=0$ by Remark \ref{rk:Tr-E}, and it follows from \eqref{eq:def-E} that
	\begin{equation}\label{eq:E=0}
		a^i_{,j}=\frac{P_0}{N}\delta^i_j \quad\text{ in }\mathbb{R}^N\setminus (Z\cup\{0\})\,,
	\end{equation}
where $a^i_{,j}$ are introduced in \eqref{eq:def-aik}.
	Thus, by the continuity of $a$ and the fact that $|Z|=0$, we infer that if $\alpha=0$ then
	\begin{equation}\label{eq:fmu-a-1}
		a=\frac{P_0}{N} (x-x_0) \quad\text{ in }\mathbb{R}^N \,.
	\end{equation}
	If $-1<\alpha<0$, then by \eqref{eq:nonnegative-Ric} and \eqref{eq:TrE=Ric=0}, it entails not only \eqref{eq:E=0} but also
	\begin{equation}\label{eq:alamx}
		a=\mu(x) x \quad\text{ in }\mathbb{R}^N\setminus (Z\cup\{0\})\,,
	\end{equation}
for some function $\mu$. Inserting \eqref{eq:alamx} into \eqref{eq:E=0} gives
	\begin{equation}\label{eq:lam-Gamma}
		\frac{P_0}{N}\delta^i_j=a^i_{,j}=\nabla_j(a^i)=\nabla_j(\mu x^i)=\mu\delta^i_j+\mu_jx^i+\mu\Gamma^{i}_{kj}x^k \,,
	\end{equation}
	where $\nabla_j$ are the covariant derivatives as introduced in Subsection \ref{sec:P-1} and the notation $\Gamma^i_{kj}$ is the Christoffel symbol associated with the metric $g=|x|^{2\alpha}\delta_e$, which is given by
	\begin{equation}\label{eq:Gamma}
		\Gamma^i_{kj}=\alpha|x|^{-2}(x^k\delta_{ij}+x^j\delta_{ik}-x^i\delta_{kj}).
	\end{equation}
	From \eqref{eq:lam-Gamma}--\eqref{eq:Gamma}, we deduce that
	$$\mu\equiv \frac{P_0}{N(\alpha+1)} \quad\text{ in }\mathbb{R}^N\setminus (Z\cup\{0\}) \,.$$
	Hence, it is seen from \eqref{eq:alamx} that when $-1<\alpha<0$, 
	\begin{equation}\label{eq:fmu-a-2}
		a= \frac{P_0}{N(\alpha+1)} x \quad\text{ in }\mathbb{R}^N\setminus\{0\}\,.
	\end{equation}

	Now, in view of \eqref{eq:def-P-v2}, \eqref{eq:fmu-a-1} and \eqref{eq:fmu-a-2}, we find that in $\mathbb{R}^N\setminus\{0\}$,
	$$\left[\frac{1}{N-1}\left(P_0v-\frac{1}{N^{N-1}}\right)\right]^{\frac{N-2}{N}}\nabla v=\frac{P_0}{N(\alpha+1)} |x-x_0|^{2\alpha}(x-x_0)$$
	for some $x_0\in\mathbb{R}^N$, with $x_0=0$ if $\alpha\neq 0$.
	This implies that for some constant $b$, 
	\begin{equation}\label{eq:fmu-v-1}
		v=\frac{1}{P_0}\left(C_N|x-x_0|^{2(\alpha+1)}+b\right)^{\frac{N}{2(N-1)}}+\frac{1}{P_0 N^{N-1}}\quad\text{ in }\mathbb{R}^N\setminus\{0\}\,,
	\end{equation} 
	where $$C_N=\frac{P_0^2(N-1)^{2-\frac{2}{N}}}{N^2(\alpha+1)^2} \,.$$
 By substituting $u=-N\log v$ with the form \eqref{eq:fmu-v-1} into equation \eqref{eq_general}, one can determine that the constant $b$ must be $0$. Thus, by the continuity of $v$ at the origin we have shown \eqref{eq:fmu-v}. Finally, from \eqref{eq:fmu-v} it is clear that the solution $u$ takes the form \eqref{eq:thm-u-HD}, with
 $$\lambda=\frac{P_0^{\frac{N}{N-1}} N^{N-1}(N-1)}{[N(\alpha+1)]^{\frac{N}{N-1}}} \,.$$
The proof is completed.
\end{proof}

\section{The linearized problem}\label{sec:linear}
In this section, we consider the linearized equation of \eqref{eq_general} at $U_\alpha$:
\begin{equation}\label{eq:Linearized}
	-\operatorname{div}(\mathcal{A}(U_\alpha)\nabla\phi)=|x|^{N\alpha}e^{U_{\alpha}} \phi \quad\text{in }\mathbb{R}^N,
\end{equation}
where $U_\alpha$ is the radial solution of \eqref{eq_general} given by \eqref{eq:U}, and $\mathcal{A}(U_\alpha)$ is as in \eqref{eq:def-B}, i.e.
\begin{equation*}
\mathcal{A}(U_\alpha)=|\nabla U_\alpha|^{N-2} \operatorname{Id} + (N-2) |\nabla U_\alpha|^{N-4}\nabla U_\alpha\otimes \nabla U_\alpha \,.
\end{equation*}
By a solution to \eqref{eq:Linearized}, we mean a function $\phi\in W_{\mathrm{loc}}^{1,N}(\mathbb{R}^N) \cap L^\infty(\mathbb{R}^N)$ such that
\begin{equation}\label{eq:def-solu-Lin}
	\int_{\Omega} \mathcal{A}(U_\alpha)(\nabla\phi \cdot \nabla\varphi)\,dx =\int_{\Omega} |x|^{N\alpha}e^{U_\alpha} \phi \varphi\,dx
\end{equation}
 for any open bounded subset $\Omega \subset \mathbb{R}^N$ and any $\varphi\in W_0^{1, N}(\Omega)$.

 Since problem \eqref{eq_general} has scaling invariance with $U_{\alpha,\lambda}$ defined in \eqref{eq:thm-u-HD} as a family of radial solutions, it is seen that the function
 \begin{equation}\label{eq:def-Z0}
	Z_\alpha:=\left.\partial_\lambda\right|_{\lambda=1} U_{\alpha,\lambda}=\frac{N-1-|x|^{\frac{N(\alpha+1)}{N-1}}}{1+|x|^{\frac{N(\alpha+1)}{N-1}}}
\end{equation}
is a solution to \eqref{eq:Linearized}. In this connection, notice that \eqref{eq_general} is also invariant under translations when $\alpha=0$, that is, for $\xi\in\mathbb{R}^N$, the function $U_0(x-\xi)$ solves \eqref{eq_general} as well. Thus, in this case it is clear that besides $Z_0$, \eqref{eq:Linearized} also has solutions 
 \begin{equation}\label{eq:Zi}
	\mathcal{Z}_i:=\left.\partial_{ \xi^i}\right|_{\xi=0}U_0(x-\xi)=\frac{N^2}{N-1}\frac{|x|^{\frac{1}{N-1}-1}x}{1+|x|^{\frac{N}{N-1}}}\,, \quad i=1,\cdots,N.
\end{equation}
Here we adopt notations as in Subsection \ref{sec:P-1}.

We say that $U_\alpha$ is \emph{non-degenerate} if the space of solutions to \eqref{eq:Linearized} is spanned only by the function $Z_\alpha$ (and the functions $\mathcal{Z}_i$ if $\alpha=0$, $1\leq i\leq N$), otherwise we say $U_\alpha$ is \emph{degenerate}. In the following theorem, we characterize \emph{all} solutions to \eqref{eq:Linearized}, which shows that $U_\alpha$ could be degenerate.

\begin{theorem}\label{thm:linearized}
	Let $N\geq 2$ and $\alpha>-1$. If $\alpha\neq \alpha_k$ for any $k\in\mathbb{N}$, then the space of solutions of \eqref{eq:Linearized} has dimension 1 and is spanned by $Z_\alpha$ defined in \eqref{eq:def-Z0}.
	
	If $\alpha=\alpha_k$ for some $k\in \mathbb{N}$, then the space of solutions of \eqref{eq:Linearized} has dimension $1+\mathcal{M}(k)$ and is spanned by $Z_\alpha$ and the functions
	\begin{equation}\label{eq:Yki}
		Z_{\alpha,i}:=\frac{|x|^{\frac{\alpha+1}{N-1}-k}\mathcal{Y}_{k,i}(x)}{1+|x|^{\frac{N(\alpha+1)}{N-1}}}\,,\quad\quad i=1,\cdots,\mathcal{M}(k)\,,
	\end{equation}
	where 
	\begin{equation}\label{eq:Multi}
		\mathcal{M}(k)=\frac{(N+2 k-2)(N+k-3)!}{(N-2)!k!}
	\end{equation}
	and $\mathcal{Y}_{k,i}$, $1\leq i\leq\mathcal{M}(k)$, form a basis of homogeneous harmonic polynomials of degree $k$ in $\mathbb{R}^N$. Here $\alpha_k$ is given by \eqref{eq:alpha-k}.
\end{theorem}

\begin{remark}
Note that the characterization result in Theorem \ref{thm:linearized} is already available in \cite[Theorem 1]{DEM12} (earlier in \cite{BP98} only for $\alpha=0$) when $N=2$, and in \cite{Tak24} when $N\geq 3$ and $\alpha=0$. Indeed, if $N=2$, then by \eqref{eq:alpha-k} and \eqref{eq:Multi},  $\{\alpha_k:k\in \mathbb{N}\setminus\{0\} \}=\mathbb{N}$ and $\mathcal{M}(k)\equiv 2$ for any $k\in \mathbb{N}\setminus\{0\}$, and in complex notation, the basis $\{\mathcal{Y}_{k,1},\mathcal{Y}_{k,1}\}$ introduced in \eqref{eq:Yki} coincides with $\{\mathrm{Re}(z^k),\mathrm{Im}(z^k)\}$. Thus, when $N=2$ and $\alpha \in \mathbb{N}$ we see that the functions $Z_\alpha$ and $Z_{\alpha,i}$ ($i=1,2$) above correspond to $Z_1$, $Z_2$ and $Z_3$ exhibited in \cite{DEM12} (with $\mu=1$ and $a=0$ therein), respectively.

On the other hand, if $\alpha=\alpha_1=0$, then $\mathcal{M}(1)=N$, and clearly, in this case the basis $\{\mathcal{Y}_{1,i}\}$ coincides  with the coordinate basis $\{x^i\}$. Thus, one can find that when $\alpha=0$, the functions $Z_{\alpha,i}$ given by \eqref{eq:Yki} coincide  with $\mathcal{Z}_i$ in \eqref{eq:Zi}, up to a constant factor. So that Theorem \ref{thm:linearized} shows the nondegeneracy of $U_0$, as was discussed in \cite{Tak24}. 
\end{remark}

We mention that the result in Theorem \ref{thm:linearized} is analogous to that obtained in \cite[Theorem 1.3]{GGN13} for the linearized problem associated with the H\'enon equation in $\mathbb{R}^N$, $N\geq 3$. We also refer to \cite{FN19,FZ22,PV21} for related studies on the nondegeneracy of radial solutions to the critical $p$-Laplace equation in $\mathbb{R}^N$, with $1<p<N$.

Notice that the problems considered in the aforementioned works naturally admit the Hilbert space structure $\mathcal{D}^{1,2}(\mathbb{R}^N)$, or $\mathcal{D}^{1,2}_{*}(\mathbb{R}^N)$ (a weighted Sobolev space  defined in \cite{PV21}). In such a setting, spectral theory and bifurcation theory can be applied to investigate geometric and/or topological properties of solutions to the original problems through the analysis of the corresponding linearized equations. 

In contrast, it is not clear which Hilbert space framework is suitable for the linearized problem \eqref{eq:Linearized} associated with Liouville-type equation \eqref{eq_general}. Nevertheless, Theorem \ref{thm:linearized} can still be proved via separation of variables and a careful Sturm--Liouville analysis.

\begin{proof}[Proof of Theorem \ref{thm:linearized}]
First, let us notice that any solution $\phi$ to \eqref{eq:Linearized} is smooth outside the origin by standard elliptic regularity theory (note that \eqref{eq:Linearized} is uniformly elliptic in each given compact subset in $\mathbb{R}^N\setminus\{0\}$). Then we can derive that in $\mathbb{R}^N\setminus\{0\}$,
\begin{align}
	&\operatorname{div}(\mathcal{A}(U_\alpha)\nabla\phi) \notag\\
	= &\operatorname{div} \left(|\nabla U_\alpha|^{N-2} \nabla \phi\right)+(N-2) \operatorname{div}\left(|\nabla U_\alpha|^{N-4}(\nabla U_\alpha\cdot \nabla \phi) \nabla U_\alpha\right) \notag\\
	 = & |\nabla U_\alpha|^{N-2} \Delta \phi+\left(\nabla|\nabla U_\alpha|^{N-2}\cdot \nabla \phi\right) +(N-2)|\nabla U_\alpha|^{N-4}(\nabla U_\alpha\cdot \nabla \phi) \Delta U_\alpha \notag\\
	&+(N-2)(\nabla U_\alpha\cdot \nabla \phi)\left(\nabla|\nabla U_\alpha|^{N-4}\cdot \nabla U_\alpha\right) \notag\\
	&+(N-2)|\nabla U_\alpha|^{N-4}(\nabla(\nabla U_\alpha\cdot \nabla \phi)\cdot \nabla U_\alpha)\,. \label{eq:divAphi}
\end{align}
Also, by direct computations we have in $\mathbb{R}^N\setminus\{0\}$,
\begin{equation}  \label{eq:cal-U-1}
\nabla U_\alpha=-\left(\frac{N^2(\alpha+1)}{N-1}\right) \frac{|x|^{\frac{N(\alpha+1)}{N-1}-2}x}{1+|x|^{\frac{N(\alpha+1)}{N-1}}}\,,
\end{equation}
\begin{align}
	&\nabla\left(|\nabla U_\alpha|^{k}\right) \notag\\
	= & \frac{k}{N-1}\left(\frac{N^2(\alpha+1)}{N-1}\right)^{k} \frac{|x|^{\frac{k(N\alpha+1)}{N-1}-2}x}{\left(1+|x|^{\frac{N(\alpha+1)}{N-1}}\right)^{k+1}}\left[1+N\alpha+(1-N)|x|^{\frac{N(\alpha+1)}{N-1}}\right], \label{eq:cal-U-2}
\end{align}
	\begin{align}
		& 
		\nabla\left(|\nabla U_\alpha|^{N-4}\right) \cdot \nabla U_\alpha \notag\\ 
		= &-\frac{N-4}{N-1}\left(\frac{N^2(\alpha+1)}{N-1}\right)^{N-3} \frac{|x|^{\frac{(N-3)(N\alpha+1)}{N-1}-1}}{\left(1+|x|^{\frac{N(\alpha+1)}{N-1}}\right)^{N-2}}\left[1+N\alpha+(1-N)|x|^{\frac{N(\alpha+1)}{N-1}}\right], \label{eq:cal-U-3}
	\end{align}
and
\begin{align}
&\Delta U_\alpha \notag\\
= &-\left(\frac{N^2(\alpha+1)}{N-1}\right) \frac{|x|^{\frac{N(\alpha+1)}{N-1}-2}}{\left(1+|x|^{\frac{N(\alpha+1)}{N-1}}\right)^2}\left[\frac{N(\alpha+1)}{N-1}+N-2+(N-2)|x|^{\frac{N(\alpha+1)}{N-1}}\right]. \label{eq:cal-U-4}
\end{align}
Hence, substituting \eqref{eq:cal-U-1}--\eqref{eq:cal-U-4} into \eqref{eq:divAphi} we deduce that in $\mathbb{R}^N\setminus\{0\}$, any solution $\phi$ to \eqref{eq:Linearized} fulfills
\begin{align}
	 	&|x|^2 \Delta \phi+ \frac{N(N-2)(\alpha+1)(x \cdot \nabla \phi)}{1+|x|^{\frac{N(\alpha+1)}{N-1}}}+(N-2)\phi_{ij}x^ix^j \notag\\ = &-\frac{N^3(\alpha+1)^2}{N-1} \frac{|x|^{\frac{N(\alpha+1)}{N-1}}\phi}{\left(1+|x|^{\frac{N(\alpha+1)}{N-1}}\right)^2}\,. \label{eq:linear}
	 \end{align}
	 
Now, let us write $\phi$ as the form
\begin{equation}\label{eq:phi}
	\phi(x)=\phi(r,\theta)=\sum_{k=0}^{\infty} \psi_k(r) Y_k(\theta)\,,
\end{equation}
where $r=|x|$, $\theta=\frac{x}{|x|}\in\mathbb{S}^{N-1}$ and
$$
\psi_k(r)=\int_{\mathbb{S}^{N-1}} \phi(r, \theta) Y_k(\theta)\,d \theta\,, \quad k\in\mathbb{N}\,.
$$
Here $Y_k(\theta)$ denotes the $k$-th spherical harmonic that satisfies
$$
-\Delta_{\mathbb{S}^{N-1}} Y_k=\mu_k Y_k,
$$
where $\Delta_{\mathbb{S}^{N-1}}$ is the Laplace--Beltrami operator on $\mathbb{S}^{N-1}$ with the standard metric and
\begin{equation*}
	\mu_k=k(N-2+k)
\end{equation*}
is the $k$-th eigenvalue of $-\Delta_{\mathbb{S}^{N-1}}$ whose multiplicity is given by $\mathcal{M}(k)$ (see \eqref{eq:Multi}).

Denote $\psi_k^{\prime}=\frac{d \psi_k}{dr}$ and $\psi_k^{\prime \prime}=\frac{d^2 \psi_k}{dr^2}$. By inserting \eqref{eq:phi} into \eqref{eq:linear}, we find that for each $k\in\mathbb{N}$ and any $r>0$,
\begin{align*}
	\psi_k^{\prime\prime}(r) & +\left(1+\frac{N(N-2)(\alpha+1)}{N-1} \frac{1}{1+r^{\frac{N(\alpha+1)}{N-1}}}\right) \frac{\psi_k^{\prime}(r)}{r} -\frac{\mu_k}{N-1} \frac{\psi_k(r)}{r^2} \notag \\
	& + \frac{N^3(\alpha+1)^2}{(N-1)^2} \frac{r^{\frac{N(\alpha+1)}{N-1}-2}}{\left(1+r^{\frac{N(\alpha+1)}{N-1}}\right)^2}\psi_k(r)=0 \,. 
\end{align*}
Then by letting
$\eta_k(r)=\psi_k(r^{\frac{1}{\alpha+1}})$, we obtain
\begin{align}
	&\eta_k^{\prime \prime}(r)+ \left(1+\frac{N(N-2)}{N-1} \frac{1}{1+r^{\frac{N}{N-1}}}\right) \frac{\eta_k^{\prime}(r)}{r}
	-\frac{\mu_k}{(N-1)(\alpha+1)^2} \frac{\eta_k(r)}{r^2} \notag\\
	& +\frac{N^3}{(N-1)^2} \frac{r^{\frac{N}{N-1}-2}}{\left(1+r^{\frac{N}{N-1}}\right)^2}\eta_k(r)=0 \,. \label{eq:etak}
\end{align}
 Exploiting Lemma \ref{lem:eta-beta} below, we see that for each $k\in\mathbb{N}$, ODE \eqref{eq:etak} admits solutions in $C^2(0,+\infty)\cap L^\infty(0,+\infty)$ if and only if $$\frac{\mu_k}{(N-1)(\alpha+1)^2} \in\{0,1\}\,,$$
which implies $k=0$ or $\alpha=\alpha_k$. Moreover, Lemma \ref{lem:eta-beta} tells us that in these two cases solutions to \eqref{eq:etak} are generated by $\bar{\eta}_0$ (see \eqref{eq:def-bareta0}) and $\bar{\eta}_1$ (see \eqref{eq:def-bareta1}), respectively. Translating these back to $\psi_k$ and recalling \eqref{eq:phi} yields the assertions in Theorem \ref{thm:linearized}. The proof is completed.
\end{proof}

In the above proof, we used the following result.
\begin{lemma}\label{lem:eta-beta}
	For a constant $\beta\geq 0$, the ODE
	\begin{equation}\label{eq:eta}
		\eta^{\prime \prime}+ \left(1+\frac{N(N-2)}{N-1} \frac{1}{1+r^{\frac{N}{N-1}}}\right) \frac{\eta^{\prime}}{r}
		-\beta \frac{\eta}{r^2}+\frac{N^3}{(N-1)^2} \frac{r^{\frac{N}{N-1}-2}}{\left(1+r^{\frac{N}{N-1}}\right)^2}\eta=0
	\end{equation}
	admits non-trivial solutions $\eta\in C^2(0,+\infty)\cap L^\infty(0,+\infty)$ if and only if $\beta\in\{0,1\}$. Moreover, the space of such solutions is spanned by 
	\begin{equation}\label{eq:def-bareta0}
		\bar{\eta}_0(r):=\frac{N-1-r^{\frac{N}{N-1}}}{1+r^{\frac{N}{N-1}}}\quad\quad\text{ when}\quad  \beta=0,
	\end{equation}
	and by 
	\begin{equation}\label{eq:def-bareta1}
	 \bar{\eta}_1(r):=\frac{r^{\frac{1}{N-1}}}{1+r^{\frac{N}{N-1}}}\quad\quad\text{ when}\quad  \beta=1,
	\end{equation}
	respectively.
\end{lemma}

 	\begin{proof}
 		First, one can directly verify that $\bar{\eta}_0$ and $\bar{\eta}_1$ are two solutions of \eqref{eq:eta} when $\beta=0$ and $\beta=1$ respectively. By the Frobenius method for second-order ODEs (see for instance \cite[Theorem 4.5]{Tes12}), we infer that there is only one branch of non-trivial solutions $\eta$ to \eqref{eq:eta} in the space $C^2(0,+\infty)\cap L^\infty(0,+\infty)$, which satisfies 
 		\begin{equation}\label{eq:asym-eta}
 			\eta=O(r^{\alpha_+}) \quad \text{as } r\to 0\,, \quad \text{ and }\quad \eta=O(r^{-\sqrt\beta}) \quad \text{as } r\to +\infty\,,
 		\end{equation}
 		where $\alpha_+$ is one of the indicial roots corresponding to the endpoint $0$: 
 $$\alpha_+=-\frac{N(N-2)}{2(N-1)} + \sqrt{\frac{N^2(N-2)^2}{4(N-1)^2}+\beta} \,.$$ 
Thus, it is seen that the space of such solutions to \eqref{eq:eta} is spanned by $\bar{\eta}_0$ when $\beta=0$ and $\bar{\eta}_1$ when $\beta=1$, respectively.

If there were a nontrivial solution $\bar{\eta}\in C^2(0,+\infty)\cap L^\infty(0,+\infty)$ for some $\beta\notin \{0,1\}$,  then a contradiction would arise. Indeed, letting
\begin{equation}\label{eq:tildeU}
 		\tilde{U}(r)=\log\frac{c_N}{\left(1+r^{\frac{N}{N-1}}\right)^N}\,,
 	\end{equation}
 	 note that \eqref{eq:eta} is equivalent to the following singular Sturm--Liouville equation:
 	 \begin{equation}\label{eq:StuLio}
 	 	-(P\eta^\prime)^\prime+Q\eta=-\beta w\eta
 	 \end{equation}
 	 where 
 	 \begin{equation*}
 	 	P=P(r)=r^{N-1} |\tilde{U}^{\prime}(r)|^{N-2}\,,\quad Q=Q(r)=-\frac{e^{\tilde{U}(r)}}{N-1} r^{N-1} \,, 
 	 \end{equation*}
  and $$w=w(r)=r^{N-3}|\tilde{U}^{\prime}(r)|^{N-2}\,.$$

 	Let $\mathcal{D}$ be the space of functions $f\in \mathrm{AC}_{\mathrm{loc}}(0,+\infty)\cap L_w^2(0,+\infty)$ such that $$Pf^\prime\in\mathrm{AC}_{\mathrm{loc}}(0,+\infty)\quad\text{and}\quad \frac{1}{w}\left[-(Pf^\prime)^\prime+Qf\right]\in L_w^2(0,+\infty)\,.$$
 	Here $\mathrm{AC}_{\mathrm{loc}}(0,+\infty)$ is the space of real valued functions on $(0,+\infty)$ that locally are absolutely continuous, and $L_w^2(0,+\infty)$ is the weighted $L^2$ space with weight $w$, containing all real valued functions $g$ on $(0,+\infty)$ such that 
 	$$\int_0^{+\infty}g^2(s) w(s)\,ds<+\infty\,.$$
 	For $f\in\mathcal{D}$, define $$\mathcal{L} f:=\frac{1}{w}\left[-(Pf^\prime)^\prime+Qf\right].$$
 	Observe that the endpoints $0$ and $+\infty$ are both limit-points for equation \eqref{eq:StuLio} in Sturm--Liouville theory (see for instance \cite[Definition 7.3.1]{Zet05}). Thus, according to the characterization in (i) of \cite[Theorem 10.4.1]{Zet05}, we see that $\mathcal{L}$ is a self-adjoint operator on $\mathcal{D}$.
 Moreover, by the assumption that $\bar{\eta}$ satisfies \eqref{eq:eta}  and using \eqref{eq:asym-eta}, we have that $-\beta$ is an eigenvalue of $\mathcal{L}$, with $\bar{\eta}\in\mathcal{D}$ being one of its eigenfunction. We claim $\beta<1$ and that $\bar{\eta}$ has at least one but finitely many zeros in $(0,+\infty)$.
 	
 	Indeed, let $\sigma_0$ be the infimum of the essential spectrum of $\mathcal{L}$. Note that $-1$ is another eigenvalue of $\mathcal{L}$ with $\bar{\eta}_1\in\mathcal{D}$ being one of its eigenfunctions. Since  $\bar{\eta}_1$ has no zeros in $(0,+\infty)$, we infer that $\sigma_0>-\infty$ by virtue of spectrum properties for Sturm--Liouville problems (see for instance (8)-(i) of \cite[Theorem 10.12.1]{Zet05}).
 	
 	If $\sigma_0=+\infty$ (i.e. when the essential spectrum of $\mathcal{L}$ is empty), then by (8)-(ii) of \cite[Theorem 10.12.1]{Zet05} we can deduce that the spectrum of $\mathcal{L}$ is bounded below with $-1$ as the minimum and $+\infty$ as the supremum. More precisely, the eigenvalues of $\mathcal{L}$ in this case are all simple and can be ordered as $\{\lambda_n:n\in\mathbb{N}\}$ fulfilling
 	\begin{equation}\label{eq:order-eigen}
 		-1=\lambda_0<\cdots<\lambda_{n-1}<\lambda_n<\cdots \,,
 	\end{equation}
 with $\lambda_n\to+\infty$ as $n\to+\infty$. Moreover, an eigenfunction of $\lambda_n$ has exactly $n$ zeros in $(0,+\infty)$.  Thus, we see $-\beta\geq\lambda_1>-1$ and the claim is true.  (note that $0$ is not an eigenvalue of $\mathcal{L}$ in $\mathcal{D}$, due to $\bar{\eta}_0\notin L_w^2(0,+\infty)$).

If $-\infty<\sigma_0<+\infty$, then we have $\sigma_0\geq0>-\beta>-1$. Indeed, by (8)-(iii) of \cite[Theorem 10.12.1]{Zet05} and the fact that $\bar{\eta}_0$ has only one zero point in $(0,+\infty)$ and is a solution to \eqref{eq:StuLio} with $\beta=0$,  we see it must hold $\sigma_0\geq0$. Also, in this case all the eigenvalues below $\sigma_0$, including $-1$ and $-\beta$, are all simple and can be ordered as in \eqref{eq:order-eigen} (accumulating at $\sigma_0$ if they are of an infinite number), with an eigenfunction of $\lambda_n$ having exactly $n$ zeros in $(0,+\infty)$. Hence, the claim is still true as in the case $\sigma_0=+\infty$.	
 	
Now, denote by $a$ and $b$ respectively the first zero and the last zero of $\bar{\eta}$ in $(0,+\infty)$. Let $c=(N-1)^\frac{N-1}{N}$ be the unique zero of $\bar{\eta}_0$ in $(0,+\infty)$. Since both $\bar{\eta}$ and $\bar{\eta}_0$ satisfy \eqref{eq:StuLio}, we deduce that 
\begin{equation}\label{eq:beta-0}
	(P\bar{\eta}^\prime)^\prime\bar{\eta}_0-(P{\bar{\eta}_0}^\prime)^\prime\bar{\eta}=\beta w\bar{\eta}\bar{\eta}_0\,.
\end{equation}
We divide into two cases,  by noting that the boundedness of $\bar{\eta}$ implies there exist sequences $r_n\to0$ and $s_n\to+\infty$ such that 
\begin{equation}\label{eq:rnsn}
	r_n\bar{\eta}^\prime(r_n)\to 0 \quad\text{and}\quad s_n\bar{\eta}^\prime(s_n)\to 0\,,\quad \text{as } n\to+\infty.
\end{equation}
 
Case 1: $a\leq c$. Assume without loss of generality that $\bar\eta>0$ on $(0,a)$ and $\bar\eta^{\prime}(a) \leq 0$. Integrating \eqref{eq:beta-0} from $r_n$ to $a$ and using \eqref{eq:asym-eta} and \eqref{eq:rnsn} yields 
\begin{align*}
0<\beta\int_{0}^a w\bar{\eta}\bar{\eta}_0 = &\, P(a)\bar{\eta}^\prime(a)\bar{\eta}_0(a)-\lim_{n\to +\infty} P(r_n)\bar{\eta}^\prime(r_n)\bar{\eta}_0(r)+\lim_{n\to +\infty}P(r_n){\bar{\eta}_0}^\prime(r_n)\bar{\eta}(r_n)\\
&- P(a){\bar{\eta}_0}^\prime(a)\bar{\eta}(a)\\
=& \, P(a)\bar{\eta}^\prime(a)\bar{\eta}_0(a)\leq 0 \,,
\end{align*}
which is a contradiction.

Case 2: $c<a$. Assume without loss of generality that $\bar\eta< 0$ on $(b,+\infty)$ and $\bar\eta^{\prime}(b) \leq 0$. Integrating \eqref{eq:beta-0} from $b$ to $s_n$ and using \eqref{eq:asym-eta} and \eqref{eq:rnsn} yields 
 	\begin{align*}
 		0<\beta\int_b^{+\infty} w\bar{\eta}\bar{\eta}_0 = &\, \lim_{n\to+\infty} P(s_n)\bar{\eta}^\prime(s_n)\bar{\eta}_0(s_n) - P(b)\bar{\eta}^\prime(b)\bar{\eta}_0(b)- \lim_{n\to +\infty}P(s_n){\bar{\eta}_0}^\prime(s_n)\bar{\eta}(s_n)\\
 		&+ P(b){\bar{\eta}_0}^\prime(b)\bar{\eta}(b)\\
 		= &\, -P(b)\bar{\eta}^\prime(b)\bar{\eta}_0(b)\leq 0 \,,
 	\end{align*}
 	giving a contradiction. 

 As a result, we have proved that when $\beta\notin \{0,1\}$, there are no non-trivial solutions to \eqref{eq:eta} in $C^2(0,+\infty)\cap L^\infty(0,+\infty)$. This completes the proof.
  \end{proof}

 	 The following is a consequence of Theorem \ref{thm:linearized}, which is analogous to \cite[Corollary 1.4]{GGN13} concerning the computation of the Morse index of radial solutions to the H\'enon equation.
 	 
 	 \begin{corollary}
 	 	For $\Lambda\in\mathbb{R}$, let $\mathrm{dim}(\Lambda)$ be the dimension of the space of solutions to the equation \begin{equation}\label{eq:Linearized-Lam}
 	 		-\operatorname{div}(\mathcal{A}(\nabla U_\alpha)\nabla\phi)=\Lambda |x|^{N\alpha}e^{U_\alpha} \phi \quad\text{ in }\mathbb{R}^N \,,
 	 	\end{equation}
  	in the sense as in \eqref{eq:def-solu-Lin}. Then $$\sum_{\Lambda<1}\mathrm{dim}(\Lambda)=\sum_{k<\mathcal{S}_{N,\alpha}}\mathcal{M}(k)\,,$$
  	where $k\in\mathbb{N}$, $\mathcal{M}(k)$ is as in \eqref{eq:Multi}, and $\mathcal{S}_{N,\alpha}$ is defined in \eqref{eq:def-S}.
 	 \end{corollary}

 \begin{proof}
 As done in \eqref{eq:phi}--\eqref{eq:etak}, by writing a solution $\phi$ to \eqref{eq:Linearized-Lam} in the form \eqref{eq:phi} and letting $\eta_k(r)=\psi_k(r^{\frac{1}{\alpha+1}})$, we deduce that in this case for each $k\in\mathbb{N}$, $\eta_k\in C^2(0,+\infty)\cap L^\infty(0,+\infty)$ satisfies
 	\begin{align}
 		&\eta_k^{\prime \prime}(r)+ \left(1+\frac{N(N-2)}{N-1} \frac{1}{1+r^{\frac{N}{N-1}}}\right) \frac{\eta_k^{\prime}(r)}{r}
 		-\frac{\mu_k}{(N-1)(\alpha+1)^2} \frac{\eta_k(r)}{r^2} \notag\\
 		=& -\frac{\Lambda N^3}{(N-1)^2} \frac{r^{\frac{N}{N-1}-2}}{\left(1+r^{\frac{N}{N-1}}\right)^2}\eta_k(r) \,. \label{eq:etak-Lam}
 	\end{align}
Notice that \eqref{eq:etak-Lam} is equivalent to the following ODE:
\begin{equation}\label{eq:StuLio-Lam}
 		-(P\eta_k^\prime)^\prime+\Lambda Q\eta_k=-\frac{\mu_k }{{(N-1)(\alpha+1)^2}}w\eta_k
 	\end{equation}
 	where $P$, $Q$ and $w$ are as in \eqref{eq:StuLio}.

It can be verified that, for each $k\in\mathbb{N}$, the function
 \begin{equation}\label{eq:solu-ak-bk}
 	\frac{r^{\mathfrak{a}_k}}{\left(1+r^{\frac{N}{N-1}}\right)^{\mathfrak{b}_k}}
 \end{equation}
is a solution of \eqref{eq:StuLio-Lam} with $\Lambda=\Lambda_k$, where 
\begin{gather*}
	\mathfrak{a}_k=\frac{\sqrt{N^2(N-2)^2+\frac{4(N-1)\mu_k}{(\alpha+1)^2}}- N(N-2)}{2(N-1)}\\
	\mathfrak{b}_k= \frac{\sqrt{N^2(N-2)^2+\frac{4(N-1)\mu_k}{(\alpha+1)^2}}+\sqrt{\frac{4(N-1)\mu_k}{(\alpha+1)^2}}- N(N-2)}{2N}\\
	\Lambda_k=\frac{\mathfrak{b}_k(N+\mathfrak{b}_k-1)}{N}\,.
\end{gather*}
 Then similarly as argued in the proof of Lemma \ref{lem:eta-beta} for \eqref{eq:StuLio} in $L_{w}^2(0,+\infty)$ via Sturm--Liouville theory, we can deduce that for each $k\in\mathbb{N}$, equation \eqref{eq:StuLio-Lam} admits solutions in $L_{-Q}^2(0,+\infty)$ when $\Lambda$ is taken from a sequence of the simple eigenvalues $\Lambda_{k,0}<\Lambda_{k,1}<\Lambda_{k,2}<\cdots$. 
 Furthermore, since the function in \eqref{eq:solu-ak-bk} has no zeros in $(0,+\infty)$, we see that $\Lambda_{k,0}=\Lambda_k$ for each $k\in\mathbb{N}$. Also, since $\bar{\eta}_0$, given in \eqref{eq:def-bareta0},  belongs to $L_{-Q}^2(0,+\infty)$ and has exactly one zero in $(0,+\infty)$ and, by Lemma \ref{lem:eta-beta}, is a solution to \eqref{eq:StuLio-Lam} when $k=0$ and $\Lambda=1$, we infer that $\Lambda_{0,1}=1$. Since $\Lambda_{k,1}$ is increasing with respect to $k$ by the characterization of the eigenvalues, $\Lambda_{k,1}\geq \Lambda_{0,1}=1$ for any $k\in\mathbb{N}$. Hence, we deduce that solutions to \eqref{eq:etak-Lam} with $\Lambda<1$ occur only for those $k\in\mathbb{N}$ such that $\Lambda=\Lambda_k<1$, which implies
 \begin{equation}\label{eq:def-S}
 	k<\mathcal{S}_{N,\alpha}:=\frac{2-N+\sqrt{(N-2)^2+4(N-1)(\alpha+1)^2}}{2}\,.
 \end{equation}
From this, the assertion is easily concluded, completing the proof. 
 \end{proof}

 \medskip

 \subsection*{Acknowledgements}
 
 Authors have been partially supported by the ``Gruppo Nazionale per l'Analisi Matematica, la Probabilit\`a e le loro Applicazioni" (GNAMPA) of the ``Istituto Nazionale di Alta Matematica" (INdAM, Italy).  P.Esposito has been partially supported by the project PNRR-M4C2-I1.1-PRIN 2022-PE1-Variational and Analytical aspects of Geometric PDEs-F53D23002690006 - Funded by the U.E.-NextGenerationEU.  Part of this work was done while P. Esposito and X. Li were visiting Dipartimento di Matematica ``Federigo Enriques" at Universit\`a degli Studi di Milano and Dipartimento di Matematica e Fisica at Universit\`a degli Studi Roma Tre, respectively, whose kind hospitality is gratefully acknowledged.
 
 The result in Theorem 1.2 of this manuscript was presented by X. Li in invited talks at ``Shape Optimization, Geometric Inequalities and Related Topics -- III" (Naples) and ``Geometric-Analytic Methods for PDEs and Applications" (Florence) in January and February 2026, respectively. He thanks Paolo Cosentino and Paolo Salani for their helpful remarks on this result during these presentations.

\end{document}